\documentclass[11pt,a4paper]{article}
\usepackage[]{geometry}
\usepackage[mathcal]{euscript}
\RequirePackage[OT1]{fontenc}
\RequirePackage[numbers]{natbib}
\RequirePackage[colorlinks,citecolor=blue,urlcolor=blue]{hyperref}
\usepackage{bm,url}                     
\usepackage{amsfonts}
\usepackage{amsthm}
\usepackage{amssymb}
\usepackage{graphicx}               
\usepackage{natbib}                 
\usepackage{colortbl}               
\usepackage{booktabs}
\usepackage[normalem]{ulem}
\usepackage{todonotes} 

\usepackage{amssymb}

\usepackage[latin1]{inputenc}
\usepackage{graphicx,color}
\newtheorem{definition}{Definition}
\newtheorem{remark}{Remark}
\newtheorem{theorem}{Theorem}
\newtheorem{lemma}{Lemma}
\newtheorem{proposition}{Proposition}
\newtheorem{corollary}{Corollary}
\newcommand{\R}[1]{\mathbb{R}^{#1}}
\usepackage{amsmath}
\usepackage{amsfonts}
\usepackage{amssymb}
\usepackage{mathtools}
\usepackage{graphicx}
\usepackage{array}
\usepackage{multirow}
\usepackage{fancyhdr}
\usepackage{makeidx}
\usepackage{hyperref}
\newcommand{\V}{\mathbb{V}}

\newcommand{\eps}{\varepsilon}

\newcommand{\ind}{\mathbb{I}}
\newcommand{\pr}{\mathbb{P}}

\newcommand{\deb}{\stackrel{\mathcal{L}}{\longrightarrow}}
\newcommand{\Vor}{\mbox{Vor}}
\newcommand{\diam}{\mbox{diam}}
\newcommand{\E}{\mathbb{E}}

\begin{document}

\begin{center}
	\Large \bf  A Generalization of the maximal-spacings in several dimensions and a convexity test.
\end{center}

\begin{center}
	\bf Catherine Aaron$^a$, Alejandro Cholaquidis$^b$, Ricardo Fraiman$^b$ 
\end{center}

\begin{center}
		$^a$  Universit\'e Blaise-Pascal Clermont II.
	$^b$ Universidad de la Rep\'ublica.\\
\end{center}

\begin{abstract}
The notion of maximal-spacing in several dimensions was introduced and studied by
Deheuvels (1983) for data uniformly distributed on the unit
cube. Later on, Janson (1987) extended the results to data uniformly distributed on any bounded
 set, and obtained a very fine result, namely, he derived
the asymptotic distribution of different maximal-spacings notions.
 These
results have been very useful in many statistical applications.

We extend Janson's results to the case where the data are generated
from a H\"older continuous density that is bounded from below and whose support is bounded.
As an application, we develop a convexity test for the support of a
distribution.
\end{abstract}

\noindent\it Key words\rm:maximal spacing, convexity test, non-parametric density estimation.

\section{Introduction}

The notion of spacings, which for one
dimensional data are just the differences between two consecutive
order statistics, have been extensively studied in the one
dimensional setting; see, e.g., the review papers  \cite{pyke65,pyke72}. Many important applications to testing and estimation problems have been derived from
the study of the asymptotic behaviour of the spacings. 
Applications to testing problems date
back to \cite{proschanpyke67}, who address the asymptotic
theory of a class of tests for Increasing Failure Rate. For
estimation problems, \cite{ranneby84} propose the maximum spacing
estimation method to estimate the
parameters of a univariate statistical model. 

In the multidimensional case, several different notions of maximal-spacing have been proposed.
Most of them are based on the nearest-neighbors balls (see for instance \cite{leo:08} or \cite{bar:09}) or on the Voronoi tessellation \cite{mill03},
(a comparison can be found in \cite{ran:05}), but they do not capture the key idea of `largest set missing the observations'. 
In contrast, this is the case with the different and global notion 
proposed in \cite{deheuvels83}, and generalized in \cite{jan87}. 

In \cite{deheuvels83}, the notion of
maximal-spacing is defined and studied for iid data uniformly distributed in $[0,1]^d$ 
as the maximal length $a$ of a cube $C=\prod[x_i,x_i+a]$, included in $[0,1]^d$ 
that does not contain any of the observations.
This notion has been extended in \cite{jan87}, in which the uniformity assumption remains
but the support of the distribution is no longer assumed to be $[0,1]^d$ but may be any compactum $S$.
Moreover, $C$ is allowed to be any 
compact and convex set. Finally, while in \cite{deheuvels83} only bounds are given, in \cite{jan87}
 the asymptotic distribution for the maximal spacing is provided.

The notion of maximal
multivariate spacing, and in particular Janson's result, has been
used to solve different statistical problems. In set estimation
(see, for instance, \cite{cf97} and \cite{crc:04} ), it is used to prove the optimality of the rates of convergence.\\

The aim of this paper is to extend Janson's result to H\"older continuous densities, 
and develop, using that extension, a test to decide whether the support is convex or not.
It is organized as follows: Section $2$ is devoted to the extension of Janson's results. A new definition 
(which includes Janson's as a particular case) is given and the associated theoretical results are presented. The proofs of these results  are given in Appendix A.
Section $3$ is dedicated to the problem of testing the convexity of the support. The corresponding  proofs are given in Appendix B. Just to mention some application of this test, let us recall that when dealing with support estimation, if the support is 
known to be convex, the convex hull of the observations provides a consistent and well studied (see for instance \cite{resul1},\cite{resul2} and \cite{walther2})
estimator of $S$ which does not require any smoothing parameter. 
Also, a convexity test can  be used to, a posteriori, select a tuning parameter. In \cite{ghl2014} a test for convexity is 
also proposed, and applied to choose the parameter of the ISOMAP (see \cite{tenen:00}) method for dimensionality reduction.
The convexity test based on Janson's extension allows us to provide an estimation of the $p$-value, whereas, in \cite{ghl2014}, it has to be
estimated via the Monte Carlo method.

\section{Main definitions and results}\label{uno}

We start by fixing some notation that will be used throughout the paper.

Given a set $S\subset \mathbb{R}^d$,
 we denote by $\partial S$, $\mathring{S}$, $\overline{S}$, $\diam(S)$, $|S|$, and $\mathcal{H}(S)$,
the boundary, interior, closure, diameter, Lebesgue measure, and convex hull of $S$, respectively. 
We denote by $\|\cdot\|$ and $\langle \cdot,\cdot \rangle$ the euclidean norm and the inner product respectively. We write $N(S, \eps)$ for the inner covering number of $S$ (i.e.: 
the minimum number of balls of radius $\eps$ and centred in $S$ required to cover $S$). Recall that if $S$ is
compact, there exists $C_S$ such that $N(S,\eps)\leq C_S \eps^{-d}$.
 We denote by $\mathcal{B}(x,\varepsilon)$ the closed ball in $\mathbb{R}^d$, of radius $\varepsilon$, centered at $x$.  We set
 $\omega_d=|\mathcal{B}(0,1)|$. 
 Given  $\lambda\in \mathbb{R}$, $A,C\subset \mathbb{R}^d$, we set $\lambda A=\{\lambda a: a\in A\}$,
 $A\oplus C=\{a+c:a\in A, c\in C\}$, and $A\ominus C=\{x:\{x\}\oplus C\subset A\}$. For the sake of simplicity, we use the notation $x+C$, instead of $\{x\}\oplus C$. If $\lambda\geq 0$, we set $A^\lambda=A\oplus \lambda\mathcal{B}(0,1)$ and $A^{-\lambda}=A\ominus \lambda \mathcal{B}(0,1)$.
Given  $A,C\subset{\mathbb R}^d$ two non-empty  compact sets,  the Hausdorff (or Pompeiu--Hausdorff) distance between them is given by 
$$d_H(A,C)=\max \bigg\{ \max_{a\in A}d(a,C),\ \max_{c\in C} d(c,A)\bigg\},$$
where  $d(a,C)=\inf\{\Vert a-c\Vert:\, c\in C\}$. 
\\

Let $A\subset \mathbb{R}^d$ be a compact convex set with $|A|=1$,
let $v$ be a vector of $\mathbb{R}^d$, and let $\alpha_A(v)$ be the constant defined in equation $(2.4)$ of \cite{jan87}:
\begin{equation} \label{alpha}
\alpha_A(v)= \frac{1}{d!}\int \dots \int \big|\mbox{Det}\big(n(y_i)\big)_{i=1}^d\big|d\omega(y_1)\dots d\omega(y_d),
\end{equation}
where $\omega$ denotes the $d-1$ dimensional Hausdorff measure, and for $y\in \partial A$, $n(y)$ denotes the exterior unit normal  vector to
 $A$ at $y$. The integral in \eqref{alpha} is over all $y_1,\dots,y_d\in \partial A$ such that $v$ is a linear combination of
 $n(y_1),\dots,n(y_d)$ with positive coefficients, and $\mbox{Det}(n(y_i))_{i=1}^d$ is the determinant of the vectors $n(y_i)$ in an orthonormal basis. 
Corollary 7.4 in \cite{jan86} proves that $\alpha_A(v)$ is  almost everywhere  independent of $v$, so that
there can be defined an $\alpha_A$ such that $\alpha_A(v)=\alpha_A$ almost everywhere.

If $A$ is the unit cube, then $\alpha_A=1$; while if $\mathcal{B}$ is the unit ball, then
\begin{equation*}
 \alpha_\mathcal{B}=\frac{1}{d!} \left( \frac{\sqrt{\pi}\Gamma\left( \frac{d}{2}+1\right)}{\Gamma\left(\frac{d+1}{2}\right)}\right)^{d-1}.
\end{equation*}

Lastly, let $U$ be a random variable such that $\pr(U\leq t)=\exp\big(-\exp(-t)\big)$.

\subsection{Janson's result and its extension}
Let $S \subset \mathbb R^d$ be a bounded set with  $|S|=1$ and $|\partial S|=0$.
 Let $\aleph_n=\{X_1,\dots,X_n\}$ be iid random vectors uniformly distributed on $S$, and $A$ a bounded convex set.
 In \cite{jan87},  the maximal-spacing is defined as
\begin{equation*}
\Delta^*(\aleph_n)=\sup \Big\{r: \ \exists x \text{ such that } x+rA\subset S\setminus \aleph_n \Big\}.
\end{equation*}

To generalize the results of \cite{jan87} to the non-uniform case,
we need to extend the definition of maximal-spacing. When the sample is drawn according to 
a probability measure $\pr_X$, we consider the probability measure of the largest set $\lambda A$ missing $\aleph_n$. 
If $|A|=1$, $\Delta^*(\aleph_n)^d$ is the Lebesgue measure of the largest set $x+r A\subset S\setminus \aleph_n$. When the sample is drawn from a non-uniform probability measure, is natural to use the same definition, replacing the Lebesgue measure by the true underling distribution $\pr_X$. If $\pr_X$ has continuous density $f$,  then $\pr_X(x+r A)\sim f(x) r^d$ for sufficiently small $r$, so one can define the maximal-spacing as the largest $r$ such that there exists $x$ with $x+\frac{r}{f(x)^{1/d}}A\subset S\setminus \aleph_n$.
\begin{definition} \label{def1}
Let $\aleph_n=\{X_1,\dots,X_n\}$ be an iid random sample of points in $\mathbb{R}^d$, 
 drawn according to a density $f$ with bounded support $S$. Let $A\subset \R{d}$ be a convex and compact set such that
 $\vert A\vert=1$ and its barycentre is the origin of $\mathbb{R}^d$. We  define
\begin{equation*}
\Delta(\aleph_n)=\sup \Big\{r: \ \exists x \text{ such that } x+\frac{r}{f(x)^{1/d}}A\subset S\setminus \aleph_n \Big\},
\end{equation*}
$$V(\aleph_n)=\Delta^d(\aleph_n),$$
and
$$U(\aleph_n)=n\Delta^d(\aleph_n)-\log(n)-(d-1)\log\big(\log(n)\big)-\log(\alpha_A).$$
\end{definition}

The following result can be found in \cite{jan87}.

\begin{theorem} \label{jan} Let $S\subset \R{d}$ be a bounded set such that $|S|=1$ and $|\partial S|=0$.
 Let $\aleph_n=\{X_1,\dots,X_n\}$  be iid random vectors uniformly distributed on $S$. Then, 
\begin{itemize}
	\item[i)] $$U(\aleph_n)\deb U\quad \text{when } n\rightarrow \infty,$$
	\item[ii)] $$\liminf_{n\rightarrow +\infty} \frac{nV(\aleph_n)-\log(n)}{\log(\log(n))}=d-1 \text{ a.s.,}$$
	\item[iii)] $$\limsup_{n\rightarrow +\infty} \frac{nV(\aleph_n)-\log(n)}{\log(\log(n))}=d+1 \text{ a.s.}$$
\end{itemize}
\end{theorem}

A rescaling extends these results to the case where $|S|\neq 1$.
\begin{corollary}  \label{jan2} Let $S\subset \R{d}$ be a bounded set such that $|\partial S|=0$ and $|S|>0$.
 Let  $\aleph_n=\{X_1,\dots,X_n\}$  be iid random vectors uniformly distributed on $S$. Then, 
\begin{itemize}
	\item[i)] $$U(\aleph_n)\deb U\quad \text{when } n\rightarrow \infty,$$
	\item[ii)] $$\liminf_{n\rightarrow +\infty} \frac{nV(\aleph_n)-\log(n)}{\log(\log(n))}=d-1 \text{ a.s.,}$$
	\item[iii)] $$\limsup_{n\rightarrow +\infty} \frac{nV(\aleph_n)-\log(n)}{\log(\log(n))}=d+1 \text{ a.s.}$$
\end{itemize}
\end{corollary}

Janson's result does not require any condition on the shape of the support $S$, 
while in our extension it will be required that the inside covering number of $\partial S$ is such that 
there exists $C_{\partial S}>0$ and $\kappa<d$ satisfying
$N(\partial S,\epsilon)\leq C_{\partial S} \epsilon^{-\kappa}$. Note that this is a very mild 
hypothesis: if $\partial S$ is smooth enough (for instance, a $\mathcal{C}^1$ $(d-1)-$dimensional manifold), it is 
fulfilled for $\kappa=d-1$. More generally, it also holds for any set $S$ with finite Minkowski content of the boundary, for $\kappa=d-1$ (see, for instance \cite{matilla}). With respect to the distribution of the sample, we require
that the density is H\"older continuous on $S$ (i.e. there exists $K_f$ and $\beta\in (0,1]$ such that for all $x,y\in S$, $|f(x)-f(y)|\leq K_f \|x-y\|^{\beta}$)
and bounded from below on it support, by a positive constant $f_0$.

This is stated in our main theorem, given below.

\begin{theorem}\label{theocont}

Let $\aleph_n=\{X_1,\dots,X_n\}$  be iid random vectors distributed according to a distribution 
$\pr_X$ whose density 
$f$ with respect to Lebesgue measure is H\"older continuous and bounded from below on its support $S$.  
Let us assume that $S$ is compact, and there exists $\kappa<d$ and $C_{\partial S}>0$ such that
$N(\partial S,\eps)\leq C_{\partial S} \eps^{-\kappa}$
Then, we have that
\begin{equation} \label{conv0}
U(\aleph_n)\deb U\quad \text{ when } n\rightarrow \infty,
\end{equation}
\begin{equation} \label{conv2}
\liminf_{n\rightarrow +\infty} \frac{nV(\aleph_n)-\log(n)}{\log(\log(n))}\geq  d-1 \text{ a.s.,}
\end{equation}
\begin{equation} \label{conv3}
\limsup_{n\rightarrow +\infty} \frac{nV(\aleph_n)-\log(n)}{\log(\log(n))}  \leq   d+1\text{ a.s.}
\end{equation}
\end{theorem}

The proof is given in Appendix A.

\section{A new test for convexity}\label{testconvex}

\subsection{The semi-parametric case} \label{semiparam}
In this section we propose, using the concept of maximal-spacing defined in Section \ref{uno},
 a consistent hypothesis test based on an iid  sample $\{X_1,\dots,X_n\}$ uniformly distributed on a compact set $S$, 
 to decide whether $S$ is convex or not.

The main idea is the following: if $S$ is convex and the sample is uniformly distributed on $S$, then $\mathcal{H}(\aleph_n)$ is a good approximation to $S$
and $|\mathcal{H}(\aleph_n)|^{-1}\ind_{\mathcal{H}(\aleph_n)}$ is a good approximation of the uniform law. As a result,
\begin{equation*}
 \tilde{\Delta}(\aleph_n)=\sup\Big\{r: \exists x \text{ such that } x+r |\mathcal{H}(\aleph_n)|^{1/d}\mathcal{B}(0,1)\subset\mathcal{H}(\aleph_n)\setminus \aleph_n \Big\}
\end{equation*}
 is a plug-in estimator of the maximal spacing and should converge to $0$. 
On the other hand, if $S$ is not convex,  $\tilde{\Delta}(\aleph_n)$ is expected to converge to a positive constant (that depends on the shape of $S$).
In order to unify notation, let us first define the maximal inner radius.

\begin{definition} 
\label{maxespset} Let $S\subset \mathbb{R}^{d}$ be a bounded set satisfying $\mathring{S} \neq \emptyset$. 
We define the maximal inner radius of $S$ as
$$\mathcal{R}(S)=\sup\big\{r:\exists x\in S \text{ such that } \mathcal{B}(x,r)\subset S\big\}.$$

\end{definition}

\begin{remark} We have
 $\Delta(\aleph_n)=\mathcal{R}(S\setminus \aleph_n)\omega_d^{1/d}|S|^{1/d}$ and
$\tilde{\Delta}(\aleph_n)=\mathcal{R}(\mathcal{H}(\aleph_n)\setminus \aleph_n)\omega_d^{1/d}|\mathcal{H}(\aleph_n)|^{1/d}$.
	\end{remark}

When testing the convexity of the support using a test statistic based on $\tilde{\Delta}(\aleph_n)$, we only obtain, in general, an upper asymptotic bound
on the test level. However, if the boundary of the support is smooth enough, we have a converging estimation of the level.
The regularity condition is the following.\\

\textbf{Condition (P):} For all $x\in \partial S$ there exists a unique vector $\xi=\xi(x)$ with $\|\xi\|=1$, such that $\langle y,\xi\rangle\leq \langle x,\xi\rangle$ for all $y\in S$, and 
$$\|\xi(x)-\xi(y)\|\leq l\|x-y\|\quad \forall \ x,y\in \partial S,$$
where $l$ is a constant. 
We will denote by $\mathcal{C}_P$ the class of convex subsets that satisfy condition \textbf{(P)}.

The convexity test and its asymptotic behaviour is given in the following theorem.

\begin{theorem} \label{testconvteo} Let $S\subset \mathbb{R}^d$ be compact with non-empty interior. 
 Let $\aleph_n=\{X_1,\ldots,X_n\}$ be a set of iid random vectors uniformly distributed on $S$.
For the following decision problem,
\begin{equation} \label{testconv}
\begin{cases}
H_0: & \text{the set } S \text{ is convex}\\
H_1: & \text{the set } S \text{ is not convex},
\end{cases}
\end{equation}
the test based on  the statistic $\tilde{V}_n=|\mathcal{H}(\aleph_n)|\omega_d\mathcal{R}\big(\mathcal{H}(\aleph_n)\setminus \aleph_n\big)^d$ with critical region given by
$$RC=\big\{\tilde{V}_n> c_{n,\gamma}\big\},$$
where
$$c_{n,\gamma} = \frac{ 1}{n} \Big( -\log\big(-\log(1-\gamma)\big) +\log(n)+(d-1)\log\big(\log(n)\big)+\log(\alpha_\mathcal{B})\Big),$$
and $\alpha_\mathcal{B}$ is the constant defined in \eqref{alpha}, is asymptotically of level less than or equal to $\gamma$. 
Moreover, if $	S\in \mathcal{C}_P$, the asymptotic level is $\gamma$. If $S$ is not convex, the test has power one
for all sufficiently large $n$.
\end{theorem}

The proof of Theorem \ref{testconvteo} is given in Appendix B.

\subsection{The non-parametric case}\label{nparam}

We now assume that we have a sample $\aleph_n=\{X_1,\dots,X_n\}$ of iid random vectors in $\mathbb{R}^d$ drawn according to an unknown density $f$.
As in the semi-parametric case, the idea is to estimate the maximal-spacing and use this estimation as a test statistic. As before, $\mathcal{H}(\aleph_n)$ is proposed as an estimator of $S$. 
To ensure that the test proposed in Theorem \ref{testconvteo} allows determining whether the support is convex or not, the density estimator should have a non-conventional behaviour: it is expected to converge toward the unknown density when the support is convex, but not when the support 
is not convex. That is why we propose the following density estimator.
\begin{definition}\label{defestimdens}
 Let $\Vor(X_i)$ be the Voronoi cell of the point $X_i$ (i.e. 
 $\Vor(X_i)=\big\{x: \|x-X_i\|=\min_{y\in \aleph_n}\|x-y\|\big\}$).
 If $K$ is a kernel function (i.e. 
 $K\geq 0$, $\int K=1$ and $\int uK(u)du=0$) and $f_{n}(x)=\frac{1}{nh_n^d}\sum K((x-X_i)/h_n)$ denotes the usual kernel density estimator, we define
\begin{equation} \label{estvordens}
 \hat{f}_n(x)=\max_{i: x\in\Vor(X_i)}f_n(X_i)\ind_{x\in\mathcal{H}(\aleph_n)}.
 \end{equation}
\end{definition}
We propose to test the convexity using the following plug-in estimator of $\Delta(\aleph_n)$:
$$\hat{\delta}\big(\mathcal{H}(\aleph_n)\setminus \aleph_n\big)=\sup \bigg\{r: \ \exists x \text{ such that } x+\frac{r}{\hat{f}_n(x)^{1/d}}A\subset \mathcal{H}(\aleph_n)\setminus \aleph_n \bigg\},$$
with $A=\omega_d^{-1/d}\mathcal{B}(0,1)$,
and reject $H_0$ (the support is convex) if $\hat{\delta}(\mathcal{H}(\aleph_n)\setminus \aleph_n)$ is sufficiently large.

The proof of Theorem \ref{mainnonconvtheo} makes use of Theorem 2.3 in \cite{gg02}. In order to apply that result, we will introduce some technical hypotheses on the kernel function.

\begin{definition}
	Let $\mathcal{K}$ be the set of kernel functions $K(u)=\phi(p(u))$, where $p$ 
	is a polynomial and $\phi$ a is bounded real function of bounded variation, such that $c_K=\int \|u\|K(u)du<\infty$, $K\geq 0$ and there exists $r_K$ and $c'_K>0$ such that $K(x)\geq c'_K$ for all $x\in\mathcal{B}(0,r_K)$. 
\end{definition}
Note that, for example, the Gaussian and the uniform kernel are in $\mathcal{K}$. 
\begin{definition} \label{standard} A set $S$ is standard if there exist positive numbers $r_0,c_S$ such 
that $|\mathcal{B}(x,r)\cap S|\geq c_S\omega_dr^d$ for all $r\leq r_0$.
We write $\mathcal{C}$ for the class of compact convex sets with non-empty interior, and $\mathcal{A}$ for the class of all compact standard sets.
\end{definition}

%
%

Finally it is also necessary to impose some conditions on the density.\\

\textbf{Condition (B):} A density $f$ with support $S$ fulfils condition B if its restriction to $S$ is Lipschitz continuous 
(i.e. there exists $k_f$ such that $\forall x,y\in S, |f(x)-f(y)|\leq k_f\|x-y\|$) and there exists 
$f_0>0$ such that $f(x)\geq f_0$ for all $x\in S$. We denote $f_1=\max_{x\in S}f(x)$.

\begin{remark} 
The condition $f(x)\geq f_0>0$ for all $x\in S$ 
	is a necessary condition to test convexity, as indeed is mentioned in \cite{ghl2014}: `...an assumption like the
	density being bounded away from zero on its support is necessary for consistent
	decision rules.' \\
\end{remark}

\begin{theorem} \label{mainnonconvtheo}
	Let $K\in \mathcal{K}$, and let $\hat{f}_n$ be as defined  in  (\ref{defestimdens}). Assume that $h_n=\mathcal{O}(n^{-\beta})$
	for some  $0<\beta<1/d$.
	Assume also that the unknown density fulfils condition B. For the following decision problem, 
	\begin{equation} 
	\begin{cases}
	H_0: & S\in \mathcal{C} \\
	H_1: & S\notin \mathcal{C},
	\end{cases}
	\end{equation}
	\begin{itemize}
		\item[a)] the test  based on the statistic 
		$\hat{V}_n=\hat{\delta}\big(\mathcal{H}(\aleph_n)\setminus \aleph_n\big)^d$ with critical region $RC=\{\hat{V}_n\geq c_{n,\gamma}\},$
			where
		$$c_{n,\gamma}=\frac{1}{n}\Big(-\log(-\log(1-\gamma))+\log(n)+(d-1)\log(\log(n))+\log(\alpha_\mathcal{B}) \Big),$$
		has an asymptotic  level less than $\gamma$.
		\item[b)] Moreover, if $S\in \mathcal{A}$ is not convex, the power is 1 for sufficiently large $n$.
	\end{itemize} 
\end{theorem}

\begin{remark}
Notice that the `optimal' kernel sequence size, $h_n=h_0n^{1/(d+4)}$, satisfies the hypothesis of 
our theorem, so that any bandwidth selection method should be suitable for testing for convexity.\\
However, in the semi-parametric case it is possible to derive the  
asymptotic behaviour for the level under regularity conditions
 on the support. In this more general setup,  we will 
not have a convergent level estimation but only a bound for the level (the price to pay for estimating the density).
The proof of Theorem \ref{mainnonconvtheo} is given in Appendix B.
\end{remark}

\subsection{Simulations}

We have performed two simulation studies to assess the behaviour of our test in the scenarios described in Sections 
\ref{semiparam} and \ref{nparam}. For the first study, the data were drawn uniformly from  
sets $S\subset \mathbb{R}^2$, and we will perform the test defined in Section \ref{semiparam} to obtain estimations of the power and the level.
 In the second study, the non-parametric case, the data can be not uniformly drawn drawn,
and we  estimate the density using the estimator given by (\ref{estvordens}). 
In this case, we consider the same sets and density as in \cite{ghl2014}. 

\subsubsection{Semi-Parametric case}

The data were generated uniformly from the sets $S_\varphi=[0,1]^2 \setminus T_\varphi$, where $T_\varphi$ is the isosceles triangle with height $1/2$ (see Figure \ref{testconvfig}) whose angle at the vertex $(1/2,1/2)$ is equal to $\varphi$. If we have a random sample from $S_\varphi$, it is clear that as $\varphi$ increases, it should be easier to detect the non-convexity of the set. The results of the simulations are summarized in Table \ref{table0}.

 \begin{table}[!ht]
\begin{center}
 \footnotesize{
\begin{tabular}{|c|c|c|c|c|c|}
\hline
\multicolumn{2}{|c|}{$\varphi=\pi/4$}&\multicolumn{2}{c|}{$\varphi=\pi/6$}&\multicolumn{2}{c|}{$\varphi=\pi/8$} \\
\hline
     n   & $\hat{\beta}$  &  n   & $\hat{\beta}$     &   n    & $\hat{\beta}$ \\
\hline
    100   & .4            & 200     & .565           & 300     &  .543   \\
 \hline
    130   & .636          & 250     &  .787          & 350     &  .679\\
\hline
    160   & .835          & 300     &  .926          & 400     &  .846\\
 \hline
    200   &  .946         & 400     &  .996          & 500     &  .976 \\
 \hline
    300   & .997          & 500     &   1             & 600    &  .997\\
 \hline
\end{tabular}}
\end{center}
 \caption{Power  estimated over 1000 replications, for different values of $\varphi$, when the sample is uniformly distributed on $[0,1]^2\setminus T_\varphi$, where $T_\varphi$ is an isosceles triangle, (see Figure \ref{testconvfig}).}\label{table0}
\end{table}

\begin{figure}[htbp]
\begin{center}
\includegraphics[scale=.3]{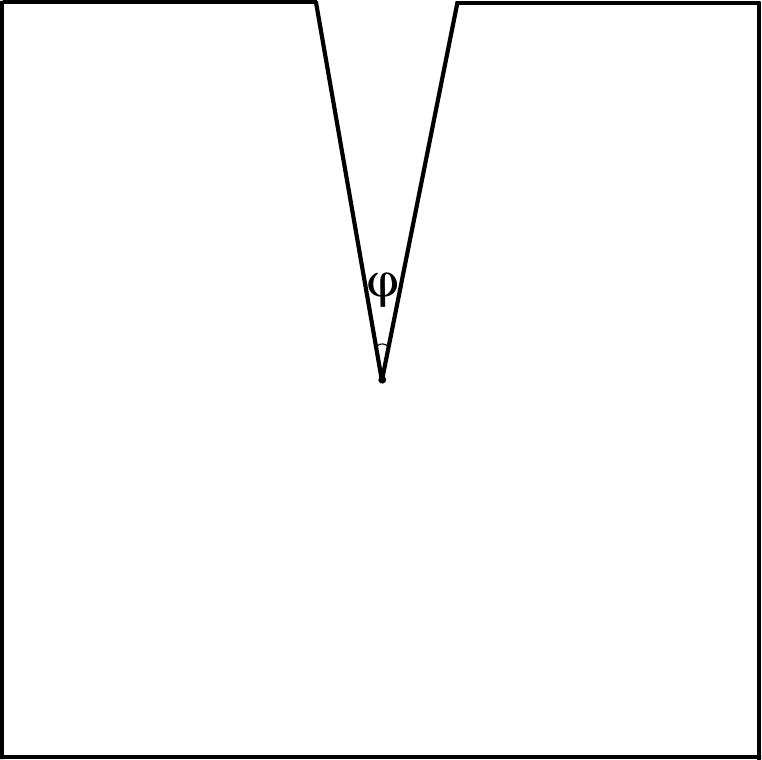} 
\caption{$[0,1]^2\setminus T_\varphi$ where $T_\varphi$ is an isosceles triangle with height $1/2$.}
\label{testconvfig}
\end{center}
\end{figure}

\subsubsection{Non-parametric case}

We performed a simulation study for the  same sets used in \cite{ghl2014}. Consider the curves $\gamma_{R,\theta}=R(\cos(\theta),\sin(\theta))$ with $\theta \in [\frac{3\pi(R-1)}{2R},\frac{3}{2}\pi]$ and the reflections of those curves along the $y$ axis (which will be denoted by  $\zeta_{R,\theta}$). We consider $\Gamma_R=T_{(0,R)}(\gamma_{R,\theta})\cup T_{(0,-R)}(\zeta_{R,\theta})$ with $\theta \in [\frac{3\pi(R-1)}{2R},\frac{3}{2}\pi]$, where $T_v$ is the translation along the vector $v$. It is easy to see that the length of every $\Gamma_R$ is $\frac{3}{2}\pi$. We will consider, for different values of $R$, the $S$-shaped sets (see the first row in Figure \ref{fig1})
\begin{small}$$S_R=T_{(0,R)}\left(\bigcup_{R-0.6\leq r\leq R+0.6}\gamma_{r,\theta}\right)\cup T_{(0,-R)}\left(\bigcup_{R-0.6\leq r\leq R+0.6}\zeta_{r,\theta}\right).$$\end{small}
Observe that when $R$ approaches infinity, the sets $S$ converge to a rectangle (which corresponds to the convex case). We have generated the data according to two different densities. The first one is the same as that considered in \cite{ghl2014}: that is, along the orthogonal direction of $\Gamma_R$, we choose a random variable with normal density (with zero mean and standard deviation $\sigma=0.15$) truncated to $0.6$ (the truncation is performed to ensure that we obtain a point in the set $S_R$). In the second case, we consider a random variable along the orthogonal direction of $\Gamma_R$ but uniformly distributed on $[-0.6,0.6]$. In Tables \ref{table1}  and \ref{table2}, we have summarized the results of the simulations, for different sample sizes (we performed the test $B=100$ times).
The results are quite encouraging and slightly better that those obtained in \cite{ghl2014} since the non-convexity is better detected 
(see Fig.~$7$ in \cite{ghl2014} for comparison) with no need for the decision rule to be calibrated.

 \begin{table}[!ht]
\begin{center}
 \footnotesize{
\begin{tabular}{|c|c|c|c|c|c|c|c|c|}
\hline
R &\multicolumn{2}{c|}{N=100}&\multicolumn{2}{c|}{N=250}&\multicolumn{2}{c|}{N=500}&\multicolumn{2}{c|}{N=1000} \\
\hline
      & np & unif             & np  & unif               &  np    & unif            &   np   & unif  \\
\hline
 1    & .13 & .44             & .55 & .99                &   1    &    1            &   1    &  1 \\
 \hline
 1.5  & .98 & 1               & 1   &  1                 &   1     &  1             &  1     & 1 \\
\hline
 3    & .38 & .24             & 1   &  1                 &    1     &  1            &   1    & 1 \\
 \hline
 6    & .08 & .09             & .41 & .66                &   1      & 1             &   1    & 1\\
 \hline
 12   & .01 & .05             & .02 & .08                &  .39    & .68            &  .98   & 1\\
 \hline
 24   &  0  & .07             & .01 & .05                &   0     & .09            &  .07   & .48\\
 \hline
$\infty$ &  0  & .04          & 0   & .09                &   0     & .04            &  .01   & .05\\
\hline
\end{tabular}}
\end{center}
 \caption{Power  estimated over $B$ replications, for different values of $R$, when the sample is uniformly distributed along the orthogonal direction of $\Gamma_R$.}\label{table1}
\end{table}
  
  \begin{table}[!ht]
\begin{center}
 \footnotesize{
\begin{tabular}{|c|c|c|c|c|c|c|c|c|}
\hline
R &\multicolumn{2}{c|}{N=100}&\multicolumn{2}{c|}{N=250}&\multicolumn{2}{c|}{N=500}&\multicolumn{2}{c|}{N=1000} \\
\hline
      & np & unif             & np  & unif               &  np    & unif            &   np   & unif  \\
\hline
 1    & 1   & 1               & 1   &  1                 &   1     &    1            &   1    &  1 \\
 \hline
 1.5  & 1   & 1               & 1   &  1                 &   1     &  1              &   1     & 1 \\
\hline
 3    & 1   & .99             & 1   &  1                 &   1     &  1              &   1    & 1 \\
 \hline
 6    & .67 & .41             & .99 &  1                 &   1     & 1               &   1    & 1\\
 \hline
 12   & .25 & .19             & .62 & .98                &  .85    & 1              &  .94   & 1\\
 \hline
 24   &  .1 & .30             & .30 & .92                &   .38   & 1              &  .48  & 1\\
 \hline
$\infty$ &  0  & .33          & .04   & .92               &   .06  & 1              &  .04  & 1\\
\hline
\end{tabular}}
\end{center}
 \caption{Power  estimated over $B$ replications, for different values of $R$, when the sample is drawn according to a truncated normal distribution along the orthogonal direction of $\Gamma_R$.}\label{table2}
\end{table}

\begin{figure}[!ht]
\begin{center}
\includegraphics[width=0.75\textwidth]{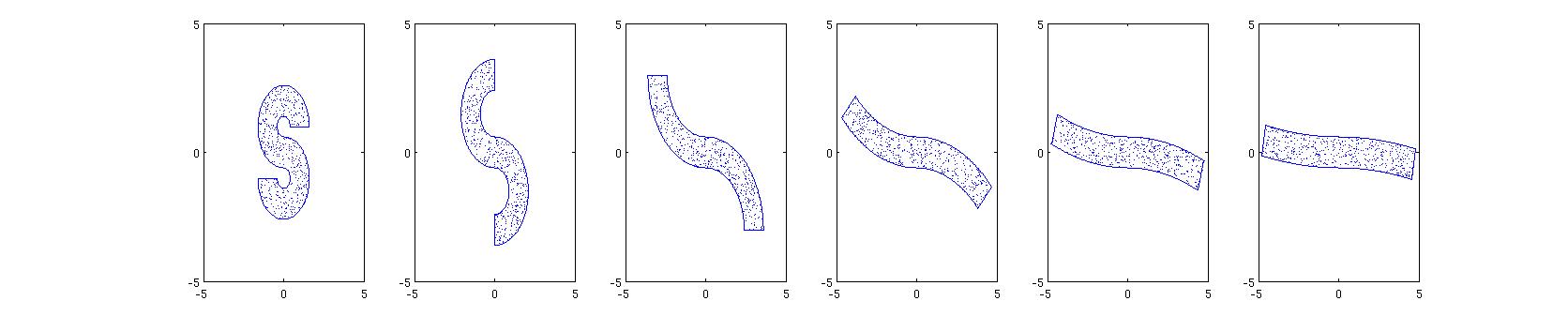}\\
 \includegraphics[width=0.75\textwidth]{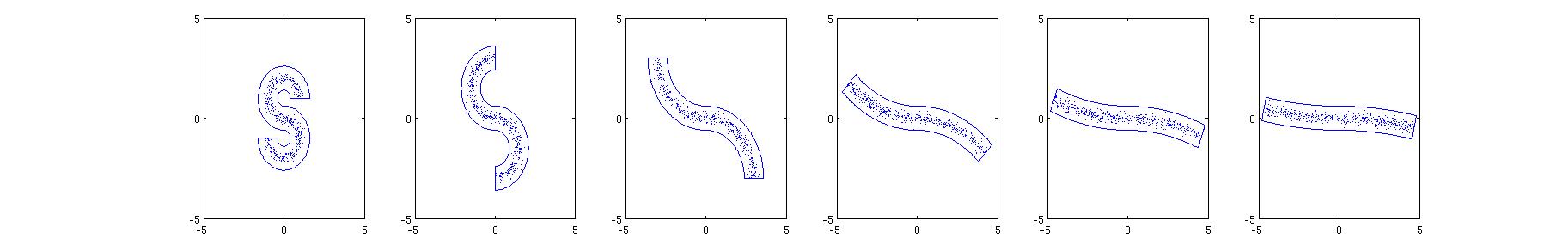}
\end{center}
\vspace{-.5cm} 
\caption{$S_R$ for different values of $R$ together with the sample drawn with a uniform radial noise (top) and 
with a truncated Gaussian noise (bottom)}.
\label{fig1}
\end{figure}
\newpage

\section{Appendix A}

In Appendix  A we proof the main result on the
generalization of the maximal-spacing, given in Theorem
\ref{theocont}. First we settle some
preliminary lemmas, then we prove a weaker version of Theorem
\ref{theocont}, for the case of piecewise constant densities on
disjoint sets. We continue by considering piecewise constant
densities, and finally we prove the result for H\"older continuous 
densities.  
\subsection{Preliminary Lemmas}

As we mentioned before, the proof of Corollary \ref{jan2} follows from a simple rescaling in 
the lemmas stated in \cite{jan87}, used to prove Theorem \ref{jan}. In particular the following rescaled lemma will be
used in this section.  

\begin{lemma}  \label{lem0} Let $S\subset \R{d}$ be a bounded set, $|S|>0$, $|\partial S|=0$, and $\aleph_n=\{X_1,\dots,X_n\}$ iid
	random vectors with uniform distribution on $S$. Then, there exists $a^S_-=a^S_-(w,n)$, $a^S_+=a^S_+(w,n)$ such that $a^S_-\rightarrow \alpha_A$ and
	$a^S_+\rightarrow \alpha_A$ as $w\rightarrow \infty$ and $w/n\rightarrow 0$,  such that if $\gamma=\frac{n}{|S|} w^{d-1}e^{-w},$
	\begin{equation}\label{eq1lem0}
	\exp(-\gamma a^S_+|S|)\leq \pr \big(nV(\aleph_n)<w\big)\leq \exp(-\gamma a^S_-|S|).
	\end{equation}
	The functions $a^S_+$ and $a^S_-$  only depend
	on the ``shape'' of $S$ (i.e. are invariant by similarity transformations).
	Without loss of generality they can be chosen such that,  for all $w'\geq w$ and $n'\geq n$:  
	$a^S_+(w',n')\leq a_+(w,n)$ and $a^S_(-w',n')\geq a_-(w,n)$.
\end{lemma}


Next we settle two lemmas whose proofs are quite
similar. The first one (Lemma \ref{lem01}) 
gives a first rough upper bound for the maximal spacing. 
The second one (Lemma \ref{lem02})
bounds a sort ``constrained'' maximal-spacing (the centre $x$ of the largest set $x+rA$ missing the observation 
is constrained to be in a ``small'' subset of the support).\\

Recall first (see \cite{amb08}) that, since $A$ is convex, there exist $\eps_0>0$ such that,
\begin{equation}\label{propconvA}
\text{for all }\eps\leq \eps_0 \text{ , } A^{-\eps}\neq
\emptyset \text{ and  }|A^{-\eps}|=|A|-\eps|\partial A|_{d-1}+o(\eps).
\end{equation}

It can also be proved easily that,
\begin{equation} \label{incconv}
\text{for all } r>0, \text{ and for all } x\in \mathcal{B}(0,\eps_0/r),\ x+(rA)^{-\|x\|}\subset rA.
\end{equation}

\begin{lemma}  \label{lem01} Let $\aleph_n=\{X_1,\dots,X_n\}$ be $iid$ random vectors in $\mathbb{R}^d$,
	with common density $f$. Assume that $f$ has bounded support $S$ and there exist $0<f_0<f_1<\infty$ such that $f_0\leq f(x)\leq f_1$
	for all $x\in S$. 
	Then,  for all $r_f>0$ such that $r_f^d>2f_1/f_0$ we have,
	\begin{equation*}
	\Delta(\aleph_n) \leq r_f \left( \frac{\log(n)}{n} \right)^{1/d} \text{ eventually almost surely. }
	\end{equation*}
\end{lemma}
\begin{proof}
	First observe that $S$ can be covered with $N(S,n^{-1/d})\leq C_Sn$ balls of radius
	$n^{-1/d}$ centered at some points $\{x_1,\dots,x_{\nu_n}\}\subset S$. Denote
	$w_n=r_f \left(\frac{ \log(n)}{n}\right)^{1/d}$ with
	$r_f^d>2f_1/f_0$. First observe that $\Delta(\aleph_n)\geq w_n
	\Leftrightarrow \exists x\in S, \text{ such that } x+w_n
	f(x)^{-1/d}A\subset S\setminus \aleph_n$, then
	$\Delta(\aleph_n)\geq w_n \Rightarrow \exists x\in S, \text{ such
		that } x+w_n f_1^{-1/d}A\subset S\setminus \aleph_n$. 
	Applying \eqref{incconv}, for sufficiently large $n$ (that is possible because 
	$n^{-1/d} \ll w_n$) we get
	
	\begin{equation}\label{ineqprobap}
	\Delta(\aleph_n)\geq w_n \Rightarrow \exists x_i \text{ such that } x_i+(w_n f_1^{-1/d}A)^{-1/n^{1/d}}\subset S\setminus
	\aleph_n.
	\end{equation}
	
	Next notice that,
	\begin{align*}
	\pr\Big(x_i+\big(w_n f_1^{-1/d}A\big)^{-1/n^{1/d}}\subset S\setminus \aleph_n\Big)
	&= \ \Bigg(1-\pr_{X} \Big( x_i+\big(w_n f_1^{-1/d}A\big)^{-1/n^{1/d}} \Big)\Bigg)^n\\
	&\leq\  \Bigg(1-f_0 \Big|\big(w_n f_1^{-1/d}A\big)^{-1/n^{1/d}}\Big| \Bigg)^n\\
	&\leq \ \Bigg(1-\Big(\frac{f_0}{f_1}w_n^d-\frac{f_0}{f_1^{\frac{d-1}{d}}}w_n^{d-1}n^{-1/d}(1+o(1))\Big) \Bigg)^n.
	\end{align*}

	The last inequality is obtained using \eqref{propconvA}. Since $w_n\gg  n^{-1/d}$, we finally get

	\begin{equation*}
	\pr\Big(x_i+\big(w_n f_1^{-1/d}A\big)^{-1/n^{1/d}}\subset S\setminus \aleph_n\Big)\leq
	\left(1-\frac{f_0}{f_1} w_n^d(1+o(1))\right)^n.
	\end{equation*}
	From this inequality and (\ref{ineqprobap}) it follows that,

	\begin{align*}
	\pr\Big(\Delta(\aleph_n)\geq r_f\big(\log(n)n/n\big)^{1/d}\Big)& \leq  N(S,n^{-1/d}) \Big(1-\frac{f_0}{f_1}w_n^d(1+o(1))\Big)^n\\
	&  \leq C_Sn \exp\Big(-\frac{f_0}{f_1}nw_n^d(1+o(1))\Big),
	\end{align*}
	
	and therefore, 
	$$\pr\Big(\Delta(\aleph_n)\geq r_f(\log(n)/n)^{1/d}\Big)\leq C_S n^{1-r_f^df_0/f_1+o(1)}.$$
	Finally, since $r_f^d>2f_1/f_0$ we have $\sum \pr\big(\Delta(\aleph_n)\geq r_f(\log(n)/n)^{1/d}\big)<\infty$. 
	Thus, the Borel-Cantelli Lemma entails that $\Delta(\aleph_n)\leq r_f(\log(n)/n)^{1/d}$ eventually almost surely.
\end{proof}

\begin{lemma}  \label{lem02} Let $\aleph_n=\{X_1,\dots,X_n\}$ be iid random vectors in $\mathbb{R}^d$ with common distribution 
	distribution $\mathbb{P}_X$ supported on a compact set $S$ and density $f$ continuous on $S$. Assume that there exists $f_0>0$ such that $f(x)\geq f_0$ $\forall x\in S$. 
	Let  $G_n$ be a sequence of sets included in S,
	with the following property: there exist $C$ such that $N(G_n,n^{-1/d})\leq Cn^{1-a}(\log(n))^{b}$ for some $a>0$ and $b>0$. 
	Let $A$ be a
	compact and convex set with $|A|=1$ such that its barycenter is
	the origin of $\mathbb{R}^d$. Let us denote
	\begin{align*}
	\Delta(\aleph_n,G_n)\ = &\ \sup \Big\{r: \ \exists x\in G_n \text{ such that } x+\frac{r}{f(x)^{1/d}}A\subset S\setminus \aleph_n \Big\},\\
	V(\aleph_n,G_n)\ = &\ \Delta^d(\aleph_n,G_n),\\
	U(\aleph_n,G_n)\ = &\ nV(\aleph_n,G_n)-\log(n)-(d-1)\log(\log(n))-\log(\alpha_A).
	\end{align*}
	
	Then, 
	$\pr \big(U(\aleph_n,G_n)\geq -\log(\log(n))\big)\rightarrow 0.$

\begin{proof}
	Let us first cover $G_n$ with $\nu_n=N(G_n,n^{-1/d})$ balls of radius $n^{-1/d}$,
	centred at some points $\{x_1,\dots,x_{\nu_n}\}$ belonging to $S$, and choose $w_n=(\frac{\log(n) +(d-2)\log(\log(n))+\log(\alpha_A)}{n})^{1/d}$ (observe that $w_n\gg(1/n)^{1/d}$).
	As in the proof of Lemma \ref{lem01} we have,
	$$\Delta(\aleph_n)\geq w_n \Leftrightarrow \exists x\in G_n, \text{ such that } x+w_n f(x)^{-1/d}A\subset S\setminus \aleph_n.$$
	which implies,
	\begin{equation*}
	\Delta(\aleph_n)\geq w_n \Rightarrow \exists x_i \ \exists x\in\mathcal{B}(x_i,n^{-1/d}),\text{ such that }
	x_i+(w_n f(x)^{-1/d}A)^{-1/n^{1/d}}\subset S\setminus \aleph_n.
	\end{equation*}
	
	Therefore, 
	\begin{equation*}
	\pr\left(x_i+(w_n f(x)^{-1/d}A)^{-1/n^{1/d}}\subset S\setminus \aleph_n\right)=\Bigg(1-\mathbb{P}_X \Big( x_i+\big(w_n f(x)^{-1/d}A\big)^{-(1/n^{1/d})} \Big)\Bigg)^n.
	\end{equation*}
	
	With rough bounds on the density,
	
	\begin{equation*}
	\mathbb{P}_X\Big( x_i+\big(w_n f(x)^{-1/d}A\big)^{-1/n^{1/d}} \Big)\geq \frac{\min_{t\in S\cap (x_i+w_nf_1^{-1/d}A)} f(t)}{\max_{t\in S\cap 
			\mathcal{B}(x_i,n^{-1/d})} f(t)} w_n^d(1+o(1)).
	\end{equation*}
	
	Since $f$ is  uniformly continuous on $S$, $A$ is bounded, $w_n\rightarrow 0$ and $n^{-1/d}\rightarrow 0$,
	for all $c<1$ there exist $n_c$ such that for all $n\geq n_c$ 
	\begin{equation*}
	\pr\left(x_i+(w_n f(x)^{-1/d}A)^{-1/n^{1/d}}\subset S\setminus \aleph_n\right)\leq
	\left(1-c w_n^d(1+o(1))\right)^n,
	\end{equation*}
	then,
	\begin{equation*}
	\pr\big(\Delta (\aleph_n,G_n)\geq w_n \big)\leq C n^{1-a}(\log n)^{b}(1-cw_n^d(1+o(1)))^n.
	\end{equation*}
	
	Taking $c=1-a/2$, we finally get that
	\begin{equation*}
	\pr \big(U(\aleph_n,G_n)\geq -\log(\log(n))\big)\leq C \alpha_A^{-1+a/2}n^{-a/2}(\log(n))^{b-(1-a/2)(d-2)}(1+o(1))\rightarrow 0.
	\end{equation*}

\end{proof}
\end{lemma}

The next lemma relates the behaviour of the maximal-spacing
for  two different densities having the same support.

\begin{lemma} \label{lem1} Let us consider $f$ and $h$, two 
	densities with compact support $S$ such that, $h(x)>h_0$ for all $x\in S$ and $\max_{x\in S}$ $|f(x)-h(x)|\leq \eps h_0$ for a given $\eps\in(0,1/2)$. Denote  by $n_0=\lfloor n(1-2\eps)\rfloor$
	and $n_1=\lceil n(1+2\eps)\rceil$ the floor and ceiling of $n(1-2\eps)$ and  $n(1+2\eps)$ respectively. 
	For any  $w\in \mathbb{R}$,   let us define 
	$w_{n,0}=\frac{w(1-2\eps-n^{-1})}{(1+\eps)}$ and $w_{n,1}=\frac{w(1-\eps)}{1+2\eps}$. Then,
	\begin{equation} \label{lem1eq1}
	\pr\big(n_0V(\mathcal{Y}_{n_0})\leq w_{n,0} \big)\Big(1-\frac{1-\eps}{n\eps}\Big)\leq \pr\big(nV(\aleph_n)\leq w \big),
	\end{equation}
	and
	\begin{equation} \label{lem1eq2}
	\pr\big(nV(\aleph_n)\leq w \big)\leq \pr\big(n_1V(\mathcal{Y}_{n_1})\leq w_{n,1} \big)\Big(1-\frac{1+2\eps+n^{-1}}{(n\eps+1)(1+\eps)} \Big)^{-1},
	\end{equation}
	where $\mathcal{Y}_{n_0}=\{Y_1,\dots,Y_{n_0}\}$ and
	$\mathcal{Y}_{n_1}=\{Y_1,\dots,Y_{n_1}\}$ are iid random vectors on $\mathbb{R}^d$, with density $h$, and
	$\aleph_n=\{X_1,\ldots,X_n\}$ are $iid$  random vectors with density $f$.
\begin{proof}
	We first prove (\ref{lem1eq1}). Observe that $X$ can be generated from the following mixture:
	with probability $p=1-\eps$, $X$ is drawn with density $h$,
	and, with probability $1-p$, $X$ is drawn with the law given
	by the density $g(x)=\frac{f(x)-h(x)(1-\eps)}{\eps}\ind_S(x)$.
	Let us denote by $N_0$ the number of points drawn according to $h$
	on $S$ and $\aleph_{N_0}^*=\{Y_1,\ldots,Y_{N_0}\}$ the
	associated sample. Let us recall that
	$$\Delta(\aleph_n)=\sup \Big\{r: \ \exists x \text{ such that } x+\frac{r}{f(x)^{1/d}}A\subset S\setminus \aleph_n \Big\}.$$
	Observe that 
	$$\sup_x h_0|f(x)/h(x)-1|\leq \sup_x h(x)|f(x)/h(x)-1|\leq h_0\eps,$$
	so $f(x)/h(x)\leq 1+\eps$. Then we have,
	$$\Delta(\aleph_n)\leq (1+\eps)^{1/d} \sup \Big\{r: \ \exists x \text{ such that } x+\frac{r}{h(x)^{1/d}}A\subset S\setminus \aleph_n \Big\}.$$
	From the inclusion $\aleph_{N_0}^*\subset \aleph_n$ we get
	$$\Delta(\aleph_n)\leq (1+\eps)^{1/d} \sup \Big\{r: \ \exists x \text{ such that } x+\frac{r}{h(x)^{1/d}}A\subset S\setminus
	\aleph_{N_0}^* \Big\},$$ and therefore $\Delta(\aleph_n)\leq
	(1+\eps)^{1/d}\Delta(\aleph_{N_0}^*)$, which entails that
	$V(\aleph_n)\leq (1+\eps)V(\aleph_{N_0}^*)$. Then, for all $w>0$,
	$$\pr\big(nV(\aleph_n)\leq w\big)\geq \pr\big((1+\eps)nV(\aleph_{N_0}^*)\leq w\big),$$
	and
	$$\pr\big(nV(\aleph_n)\leq w\big) \geq \pr\Big(\big((1+\eps)nV(\aleph_{N_0}^*)\leq w\big) \cap \big(N_0\geq n_0\big)\Big).$$

	For $N_0\geq n_0$, let us denote by $\mathcal{Y}_{n_0}=\{Y_1,\ldots,Y_{n_0}\}$ the $n_0$ first values of $\aleph_{N_0}^*$. Clearly we have
	$V(\aleph_{N_0}^*)\leq V(\mathcal{Y}_{n_0})$ so,
	$$\pr^{N_0\geq n_0}\left((1+\eps)nV(\aleph_{N_0}^*)\leq w\right)\geq \pr\Big((1+\eps)nV(\mathcal{Y}_{n_0})\leq w\Big),$$
	where $\pr^{N_0\geq n_0}$ denotes the conditional probability given ${N_0\geq n_0}$.
	Therefore,
	$$\pr\big(nV(\aleph_n)\leq w\big) \geq \pr \Big(n_0 V(\mathcal{Y}_{n_0})\leq \frac{w n_0}{(1+\eps) n} \Big) \pr (N_0\geq n_0).$$
	
	On the other hand, since   $N_0\sim Bin(n,1-\eps)$, we obtain,
	$$\pr(N_0<n_0)=\pr\Big(N_0-(1-\eps )n<n_0-(1-\eps )n\Big)  \leq \pr(N_0-(1-\eps)n\leq -\eps n).$$
	Since $n_0=\lfloor n(1-2\eps)\rfloor$, $n_0-n(1-\eps)\leq-\eps n$, which together with Chebyshev's inequality entails that 
	$$\pr(N_0<n_0)\leq \frac{n\eps(1-\eps)}{n^2\eps^2}=\frac{(1-\eps)}{n\eps},$$
	and then,
	$$\pr(N_0\geq n_0)\geq 1-\frac{(1-\eps)}{n\eps}.$$
	Let us denote by  $w_{n,0}=\frac{w(1-2\eps-n^{-1})}{(1+\eps)}$. Since
	$n(1-2\eps)-1\leq n_0$ we have 
	$w_{n,0}\leq \frac{w n_0}{(1+\eps) n}$, from where it follows that
	$$\pr\big(nV(\aleph_n)\leq w\big) \geq \pr \Big(n_0 V(\mathcal{Y}_{n_0})\leq w_{n,0} \Big) \left(1-\frac{1-\eps}{n\eps} \right).$$
	
	Equation (\ref{lem1eq2}) is proved in the same way.
	We just provide a sketch of the proof.  The
	key point for the proof of \eqref{lem1eq1} was to think the law of a random variable $Y$ drawn with the density $h$ as the following mixture:
	with probability $p=\frac{1}{1+\eps}$, $Y$ as a
	random variable with  density $f$, and, with
	probability $1-p$, $Y$ is drawn with density $g(x)=\frac{h(x)(1+\eps)-f(x)}{\eps} \ind_S(x)$. Next,
	we consider a sample $\mathcal{Y}_{n_1}=\{Y_1,\ldots,Y_{n_1}\}$ of
	iid copies of  $Y$, (that follows a law given by $h$). Denote by $N$
	the number of the points that drawn according to the
	density $f$ and $\mathcal{Y}_N^*=\{X_1,\ldots X_N\}$ these points.
	The rest of the proof follows using the same argument to prove
	(\ref{lem1eq1}).
\end{proof}
\end{lemma}
\subsection{Uniform mixture on disjoint supports}

\begin{proposition} \label{theounif} Let $E_1,\dots,E_k$ be subsets of $\mathbb{R}^d$ such that for all $i\neq j \Rightarrow \overline{E_i}\cap\overline{E_j}=\emptyset$,
	and $0<|E_i|<\infty$ for all $i$. Let $\aleph_n=\{X_1,\dots,X_n\}$ be iid random vectors in $S=\cup_i E_i$ with 
	density
	$$f(x)=\sum_{i=1}^k p_i\ind_{E_i}(x),$$
	where $p_1,\ldots, p_k$ are positive real numbers. Then,
	$$U(\aleph_n)\deb U\quad \text{when } n\rightarrow \infty.$$
\begin{proof} First let us introduce some notation,
for $i=1,\dots,k$
	\begin{itemize}
		\item $N_i=\#\{\aleph_n\cap E_i\}$ denotes the number of data  points in $E_i$. Notice that $N_i\sim Bin(n,p_i|E_i|)$.
		\item $\aleph^i_{N_i}=\{X_{i_1},\dots,X_{i_{N_i}}\}$ denotes the subsample of $\aleph_n$ that falls in $E_i$. 
		Observe that they all are uniformly distributed.
		\item $a_i=p_i|E_i|$ for $i=1,\dots,k$, that fulfils $\sum a_i=1$, $a_0=\min_i a_i$, $A_0=\max_i a_i$  and $C=\sum \frac{1-a_i}{a_i}$.
		\item $\eps_{n,i}=\frac{N_i-a_in}{na_i}$.
	\end{itemize}
	Since the support of $f$ is $\cup_i\overline{E_i}$, and
	by assumption  $i\neq j \Rightarrow \overline{E_i}\cap
	\overline{E_j}=\emptyset$, we have 
	$$\Delta(\aleph_n)=\sup\Big\{r: \ \exists x \exists i \text{ such that } x+\frac{r}{p_i^{1/d}}A\subset E_i\setminus \aleph_n\Big\},$$
	so
	\begin{equation} \label{eq1}
	\Delta(\aleph_n)=\max_i\sup\Big\{r: \  \exists x  \text{ such that
	} x+\frac{r|E_i|^{1/d}}{(|E_i|p_i)^{1/d}}A\subset E_i \setminus
	\aleph^i_{N_i}\Big\},
	\end{equation}
	while
	\begin{equation} \label{eq2}
	\Delta(\aleph^i_{N_i})=\sup\Big\{r': \ \exists x \in E_i\text{
		such that } x+r'|E_i|^{1/d}A\subset E_i\setminus
	\aleph^i_{N_i}\Big\}.
	\end{equation}
	From (\ref{eq1}) and (\ref{eq2}) we derive that
	$$\Delta(\aleph_n)= \max_i\Big\{(|E_i|p_i)^{1/d}\Delta(\aleph^i_{N_i})\Big\} \text { and } V(\aleph_n)=\max_i\Big\{|E_i|p_iV(\aleph^i_{N_i})\Big\},$$
	which entails that
	\begin{equation*}
	\pr(n V(\aleph_n)\leq w)=\prod_i\pr\left(N_i V(\aleph^i_{N_i})\leq \frac{wN_i}{a_in}\right).
	\end{equation*}
	Let  $\pr^{\overrightarrow{n}}(A)=\pr(A|N_1=n_1,\dots,N_k=n_k)$ stand for the conditional
	probability given the number of points that fall in each $E_i$.    We
	have that,
	$$\pr^{\overrightarrow{n}}\big(n V(\aleph_n)\leq w\big)= \prod_{i=1}^k \pr^{\overrightarrow{n}}\big(n|E_i|p_i
	V(\aleph^i_{n_i})\leq w\big)= \prod_{i=1}^k
	\pr^{\overrightarrow{n}}\left(n_iV(\aleph^i_{n_i})\leq
	\frac{wn_i}{n|E_i|p_i}\right).$$
	Now, taking $w_{n,i}=\frac{wn_i}{n|E_i|p_i}$, $\gamma_{n,i}=\frac{n_i
		w_{n,i}^{d-1}e^{-w_{n,i}}}{|E_i|}$ and applying Lemma \ref{lem0} we obtain,
	$$ \exp\bigg(-\sum_{i=1}^k \gamma_{n,i} a_+^{E_i} |E_i| \bigg)\leq \pr^{\overrightarrow{n}}\big(n V(\aleph_n)\leq w\big)\leq \exp\bigg(-\sum_{i=1}^k \gamma_{n,i} a_-^{E_i} |E_i| \bigg).$$
	On the other hand,
	\begin{align*}
	\sum_{i=1}^k \gamma_{n,i} a_+^{E_i} |E_i|=&\sum_{i=1}^k n_i\left(\frac{wn_i}{n|E_i|p_i} \right)^{d-1}\exp \left(-\frac{wn_i}{n|E_i|p_i} \right)a_+^{E_i}\\
	=&\sum_{i=1}^k n_i w^{d-1}(1+\eps_i)^{d-1}\exp (-w(1+\eps_i))a_+^{E_i}(w_{n,i},n_i).
	\end{align*}
	Let  $\eps_n=\max_i |\eps_{n,i}|$ and $\eps_{a_{+}}=\max_i \frac{|a_+^{E_i}(w_{n,i},n_i)-\alpha_A|}{\alpha_A}$, then we have
	
	\begin{equation*}
	\sum_{i=1}^k \gamma_{n,i} a_+^{E_i} |E_i| \leq nw^{d-1}\exp(-w)\alpha_A (1+\eps_n)^{d-1} \exp(w\eps)(1+\eps_{a_{+}}).
	\end{equation*}
	Taking $w=w_n=x+\log(n) +(d-1)\log(\log(n)) +\log (\alpha_A)$, we obtain
	that $nV\leq w \Leftrightarrow U\leq x$, which implies that
	\begin{equation*}
	\pr^{\overrightarrow{n}}\big(U(\aleph_n)\leq x\big) \geq \exp \left(-e^{-x}(1+\eps)^{d-1} \exp(\log(n) \eps_n)(1+\eps_{a_{+}})(1+o_n(1))\right).
	\end{equation*}
	In the same way it can be proved,
	\begin{equation*}
	\pr^{\overrightarrow{n}}\big(U(\aleph_n)\leq x\big) \leq \exp \left(-e^{-x}(1-\eps_n)^{d-1} \exp(-\log(n) \eps_n)(1-\eps_{a_{-}})(1+o_n(1))\right).
	\end{equation*}
	where we denoted
	$\eps_{a_{-}}=\max_i \frac{|a_-^{E_i}(w_{n,i},n_i)-\alpha_A|}{\alpha_A}$.

	Suppose that $\eps_n=\max |\eps_{n,i}|\leq 1/\log(n)^2$, then, if $n\geq 5$, $a_0
	n/2\leq N_i\leq n$ for all $i$, which imply that for all $i$, 
	$w_{n,i}\geq \log(n) a_0/(2A_0)\rightarrow \infty$ and $w_{n,i}/n\leq \big(x+\log(n) +(d-1)\log(\log(n))+\log(\alpha_A)\big)/(na_0)\rightarrow 0$.
	Then $\eps_{a_-}$ and $\eps_{a_+}$ converges to $0$, 
	according to Lemma \ref{lem1}. Therefore 
	\begin{equation} \label{inter1}
	\pr^{\eps_n\leq 1/\log(n)^2}\big(U(\aleph_n)\leq x\big) \rightarrow  \exp (-\exp(-x))\ \ \text{ when } n\rightarrow \infty.
	\end{equation}
	Since 
	\begin{equation*}
	\pr\Big(\max_i|\eps_{n,i}|\geq
	\frac{1}{\log(n)^2}\Big)=\pr\Big(\bigcup_i |\eps_{n,i}|\geq
	\frac{1}{\log(n)^2}\Big) \leq \sum_{i=1}^k \pr\Big(|\eps_{n,i}|\geq
	\frac{1}{\log(n)^2}\Big),
	\end{equation*}
	from Chebyshev's inequality we obtain
	$$\pr \Big(|\eps_{n,i}|\geq \frac{1}{\log(n)^2}\Big)\leq \log(n)^4 \V(\eps_{n,i}^2)\quad \text{ where } \quad  \V(\eps_{n,i}^2)=\frac{1-a_i}{na_i},$$
	and therefore
	\begin{equation} \label{inter2}
	\pr\Big(\eps_n\geq \frac{1}{\log(n)^2}\Big)\leq C\frac{\big(\log(n)\big)^4}{n}.
	\end{equation}
	Finally, from equations (\ref{inter1}) and (\ref{inter2}) we get,
	$$\pr\big(U(\aleph_n)\leq x\big) \rightarrow  \exp (-\exp(-x))\text{ when } n\rightarrow \infty,$$
	which concludes the proof.
\end{proof}
\end{proposition}

\subsection{Uniform mixture}
\begin{proposition} \label{theounif2}
	Let $E_1,\dots,E_k$ be subsets of $\mathbb{R}^d$ such that,
	\begin{itemize}
		\item[1)] $i\neq j \Rightarrow |\overline{E_i}\cap\overline{E_j}|=0$.
		\item[2)] $0<|E_i|<\infty$ for $i=1,\dots,k$.
		\item[3)] There exists a constant $C>0$ such that, for all $i$, for all $\eps>0$, $N(\partial E_i,\eps)\leq C\eps^{-d+1}$.
	\end{itemize}
	Let $\aleph_n=\{X_1,\dots,X_n\}$ be iid random vectors with density 
	$$f(x)=\sum_{i=1}^k p_i \ind_{\mathring{E_i}},$$
	where $p_1,\dots,p_k$ are positive real numbers. If there exist constants $r_0>0$ and $c>1-1/d$ such that,
	for all $r\leq r_0$ and all $x\in \cup_i\mathring{E_i}$,
	\begin{equation*}
	\frac{\min_{t\in S\cap B(x,r)}f(t)}{\max_{t\in S\cap B(x,r)}f(t)}\geq c,
	\end{equation*}
	then,
	$$U(\aleph_n)\deb U\quad \text{ when } \quad n\rightarrow \infty.$$
\begin{proof} We start by introducing some definitions and
	notation. 
	\begin{align*}
	\mathring{\Delta}(\aleph_n)\ =&\ \sup \Big\{r: \ \exists x \exists i, \ \text{ such that } x+\frac{r}{f(x)^{1/d}}A\subset \mathring{E_i}\setminus \aleph_n \Big\},\\
	\mathring{V}(\aleph_n)\ = &\ \mathring{\Delta}^d(\aleph_n),\\
	\mathring{U}(\aleph_n)\ = & \ n\mathring{V}(\aleph_n)-\log(n)-(d-1)\log\big(\log(n)\big)-\log(\alpha_A)
	\end{align*}
	
	With the same ideas used to prove Proposition 
	\ref{theounif} (and the fact that $|E_i|=|\mathring{E}_i|$) it follows that 
	$\mathring{U}\big(\aleph_n\big)\deb U$.
	Let $F_n(x)=\pr\big(\mathring{U}\big(\aleph_n\big)\leq x\big)$.
	Clearly $U(\aleph_n)\geq \mathring{U}\big(\aleph_n\big)$, 
	and therefore 
	\begin{equation}\label{eqU0}
	\pr\big( U(\aleph_n)\leq x\big)\leq F_n(x)\rightarrow \exp(-\exp(-x)).
	\end{equation}
	
	In order to prove the other inequality let us define
	\begin{itemize}
		\item $G=\bigcup_{i,j} \big(\overline{E_i}\cap\overline{E_j}\big)$.
		\item $p_0=\min_{i}p_i$.
		\item $\rho_A=\max_{x\in A}\|x\|$.
		\item $\rho_n=(r_f \rho_A/p_0^{1/d}) \big(\log(n)/n\big)^{1/d}$ with $r_f$ such
		that $\Delta(\aleph_n)\leq r_f\big(\log(n)/n\big)^{1/d}$
		eventually almost surely (whose existence follows from Lemma \ref{lem01}). Notice that condition 3) ensures that 
		$N(G^{\rho_n},n^{-1/d})=\mathcal{O}(n^{1-1/d}(\log(n))^{1/d}).$
		\item $\displaystyle \Delta\big(\aleph_n, S\setminus G^{\rho_n}\big)=\sup\Big\{r: \ \exists x\in S\setminus G^{\rho_n} \text{ such that } x+\frac{r}{f(x)^{1/d}}A\subset S\setminus \aleph_n \Big\}$.
		\item $\displaystyle \Delta\big(\aleph_n,  G^{\rho_n}\big)=\sup\Big\{r: \ \exists x\in G^{\rho_n} \text{ such that } x+\frac{r}{f(x)^{1/d}}A\subset S\setminus \aleph_n \Big\}$.
	\end{itemize}
	Clearly we have that,
	\begin{equation} \label{ineqdelta0}
	\Delta(\aleph_n)=\max\Big\{\Delta\big(\aleph_n, S\setminus G^{\rho_n}\big),\Delta\big(\aleph_n,  G^{\rho_n}\big)\Big\}.
	\end{equation}
	
	For the chosen $\rho_n$, we are going to prove that,
	\begin{equation} \label{ineqdelta1}
	\mathring{\Delta}\big(\aleph_n\big) \geq \Delta\big(\aleph_n, S\setminus G^{\rho_n}\big) \text{ eventually almost surely.}
	\end{equation}
	Let us suppose first that  $\Delta\big(\aleph_n, S\setminus G^{\rho_n}\big)\leq r_f (\log(n)/n)^{1/d}$ (which holds, e.a.s., due to Lemma \ref{lem01}) then
	$$\text{for all } \eps>0 \text{ there exists }  x_{\eps}\in S\setminus G^{\rho_n} \text{ such that } x_{\eps}+\frac{\Delta\big(\aleph_n, S\setminus G^{\rho_n}\big)-\eps}{f(x_{\eps})^{1/d}}A\subset S\setminus\aleph_n,$$
	and
	$$x_{\eps}+\frac{\Delta\big(\aleph_n, S\setminus G^{\rho_n}\big)-\eps}{f(x_{\eps})^{1/d}}A\subset
	\mathcal{B}\bigg(x_{\eps},\rho_A\frac{r_f(\log(n)/n)^{1/d}-\eps}{p_0^{1/d}}\bigg)\subset \mathcal{B}(x_{\eps},\rho_n).$$
	
	From $d(x_{\eps},G)\geq\rho_n$ we
	get
	$x_{\eps}+\frac{\Delta\big(\aleph_n, S\setminus G^{\rho_n}\big)-\eps}{f(x_{\eps})^{1/d}}A\subset \bigcup_i \mathring{E}_i \setminus\aleph_n.$
	Then, for all $\eps>0$,
	$\mathring{\Delta}\big(\aleph_n\big)\geq \Delta\big(\aleph_n,
	S\setminus G^{\rho_n}\big)-\eps$ e.a.s., which concludes the proof of \eqref{ineqdelta1}.
	
	By \eqref{ineqdelta0}, introducing $U(\aleph_n,G^{\rho_n})=n\Delta(\aleph_n,G^{\rho_n})^d-\log(n)-(d-1)\log(\log(n))-\log(\alpha_A)$, 
	one can bound $\pr(U(\aleph_n)\geq x)$ for all $x$, as follows:  
	$$ \left \{
	\begin{array}{l l}
	\pr(U(\aleph_n)\geq x)\leq \pr(\mathring{U}(\aleph_n)\geq x) + \pr(U(\aleph_n,G^{\rho_n})\geq -\log(\log(n))) & \text{ if } x\geq -\log(\log(n)) \\
	\pr(U(\aleph_n)\geq x)\leq 1 & \text{ if } x\leq -\log(\log(n))\\
	\end{array}
	\right.
	$$
	By Lemma \ref{lem02} we obtain:
	\begin{equation}\label{eqU1}
	\left \{
	\begin{array}{l l}
	\pr(U(\aleph_n)\leq x)\geq F_n(x)+o(1) & \text{ if } x\geq -\log(\log(n)) \\
	\pr(U(\aleph_n)\leq x)\geq 0 & \text{ if } x\leq -\log(\log(n)).\\
	\end{array}
	\right.
	\end{equation}
	
	Finally using \eqref{eqU0}, \eqref{eqU1} and that  $\mathbb{P}\big(U\leq -\log(\log(n))\big)=1/n\rightarrow 0$ we conclude the proof.
\end{proof}
\end{proposition}

\subsection*{Proof of Theorem \ref{theocont}}

	Let $c_n=(\log(n)/n)^{\frac{1}{3d}}$. Take a ``mesh'' of $\R{d}$ with small squares of side $c_n$,
	$$\prod_{i=1}^d\Big[k_ic_n\ ,(k_i+1)c_n\Big]\quad \text{ with } k_i\in \mathbb{N},$$
	and denote by $m_n\leq |S|c_n^{-d}$ the number of  these squares $\{C_1,\dots,C_{m_n}\}$ that are included in $S$.

	Like in the proof of Proposition \ref{theounif2} let us denote,
	\begin{align*}
	\mathring{\Delta}(\aleph_n)= & \ \sup \Big\{r: \ \exists x \exists i, \ \text{ such that } x+\frac{r}{f(x)^{1/d}}A\subset \mathring{C_i}\setminus \aleph_n \Big\},\\
	\mathring{V}(\aleph_n)= & \ \mathring{\Delta}^d(\aleph_n),\\
	\mathring{U}(\aleph_n)= & \ n \mathring{V}(\aleph_n)-\log(n)-(d-1)\log\big(\log(n)\big)-\log(\alpha_A).
	\end{align*}
	From the inclusion $\bigcup_{i=1}^{m_n} \mathring{C}_i\subset S$ it follows that, $\pr\big(U(\aleph_n)\leq x\big)\leq \pr\Big(\mathring{U}(\aleph_n)\leq x\Big)$.\\
	Like in the proof of Proposition \ref{theounif} let us  denote, for $i=1,\dots,m_n$.
	\begin{itemize}
		\item $N_i=\#\{\aleph_n\cap C_i\}$, 
		\item $a_i=\int_{C_i} f(t)dt$; $a_0=\min_i a_i$; $A_0=\max_i a_i$ and $C=\sum \frac{1-a_i}{a_i}$. Observe that $\sum a_i=1$ and $a_0\geq f_0c_n^d$.
		\item $\aleph^i_{N_i}=\{X_{i_1},\dots,X_{i_{N_i}}\}$, the subsample of $\aleph_n$ that belongs to $C_i$.
		Observe that $X_{i_j}$ for $j=1,\dots,N_i$ has density $f_i(x)=\big(f(x)/a_i\big)\ind_{C_i}(x)$.
		\item $\eps_{n,i}=\frac{N_i-a_in}{na_i}$, $\eps_n=\max_{i} \eps_i$.
	\end{itemize}
	We start with some asymptotic properties about $\eps_{n,i}$ and $\eps_n$. If we bound $|a_i|\geq f_0|c_n|^{d}$ and apply Hoeffding's 
	inequality we get
	
	$$\pr(\log(n) |\eps_{n,i}|\geq t )\leq 2 \exp\Big(-2t^2 f_0^2 (\log(n))^{-4/3}n^{1/3}\Big),$$
	and,
	$$\pr(\log(n) |\eps_n|\geq t )\leq \frac{2 |S| n^{1/3}}{(\log n)^{1/3}} \exp\Big(-2t^2 f_0^2 (\log(n))^{-4/3}n^{1/3}\Big).$$
	Borel-Cantelli Lemma entails $(\log(n))\eps_n \stackrel{a.s.}{\rightarrow} 0.$
	Then, with probability $1$, for $n$ large enough,
	
	\begin{equation}\label{boundNi}
	\frac{f_0}{2}(\log n)^{1/3}n^{2/3}\leq N_i \leq 2f_1 (\log n)^{1/3}n^{2/3} \quad \text{ for }i=1,\dots,m_n.
	\end{equation}
	In what follows $n$ is large enough so that \eqref{boundNi} is fulfilled.

	Proceeding exactly as in the proof of Proposition \ref{theounif} we can derive that
	
	\begin{equation*}
	\mathring{\Delta}(\aleph_n)=\max_i\sup\Big\{r: \ \exists x \text{ such that } x+\frac{ra_i^{1/d}}{(a_if_i(x))^{1/d}}A\subset \mathring{C}_i\setminus \aleph^i_{N_i}\Big\},
	\end{equation*}
	and therefore
	$$\mathring{\Delta}(\aleph_n)= \max_i\Big\{a_i^{1/d}\Delta(\aleph^i_{N_i})\Big\} \quad \text{ and } \quad V(\aleph_n)=\max_i\Big\{a_iV(\aleph^i_{N_i})\Big\}.$$

	First we to bound $\pr(U_n(\aleph_n)\geq x)$ from above. As in Proposition \ref{theounif} 
	\begin{equation}\label{t5i1}
	\pr\big(nV(\aleph_n)\leq w_n\big)\leq \pr\big(n\mathring{V}(\aleph_n)\leq w_n\big)=\prod_{i=1}^{m_n} \pr\left(N_i\Delta^d(\aleph_{N_i})\leq \frac{w_nN_i}{a_in}\right).
	\end{equation}

	At any of the small squares $C_i$, by H\"{o}lder continuity, the density is close to the uniform density, that will 
	allow us to apply  Lemma \ref{lem1} with $h=1/|C_i|\ind_{C_i}$.
	More precisely: for all $i$ an for all $y\in C_i$,
	
	$$\Big|f_i(y)|C_i|-1\Big|=\Big|\frac{f(y)}{a_i}|C_i|-1\Big|=
	\frac{1}{a_i}\Big|\int_{C_i}f(y)dt-\int_{C_i}f(t)dt\Big|\leq \frac{1}{a_i}K_f\int_{C_i}|y-t|^{\beta}dt.
	$$
	Since $|y-t|\leq \sqrt{d}c_n$, if we denote $A_f=K_ff_0^{-1}d^{\beta/2}$ we derive that
	\begin{equation*}
	\Big|f_i(y)|C_i|-1\Big|\leq \frac{1}{a_i}K_f\sqrt{d}^{\beta}c_n^{d+\beta}\leq  A_fc_n^{\beta} \quad \forall y\in C_i.
	\end{equation*}
	
	Let $N_i'=\lceil N_i(1+2 A_fc_n^{\beta})
	\rceil$, $w_n'=w_n\frac{1- A_fc_n^{\beta}}{1+ 2A_fc_n^{\beta}}$ and $\mathcal{Y}_{N_i'}$
	a sample of $N_i'$ variables uniformly drawn on $C_i$, then Lemma \ref{lem1} implies 
	that
	\begin{equation}\label {lem1ap1}
	\pr\left(N_i\Delta^d(\aleph_{N_i})\leq \frac{w_nN_i}{a_in}\right)\leq \pr\left(N_i'\Delta^d(\mathcal{Y}_{N'_i})\leq \frac{w_n'N_i'}{a_in} \right)
	\left(1-\frac{1+2 A_fc_n^{\beta}+N_i^{-1}}{(N_i A_fc_n^{\beta})(1+ A_fc_n^{\beta})} \right)^{-1}.
	\end{equation}
	On the other hand, by \eqref{boundNi}, with probability one, for $n$ large enough we have that,
	\begin{equation} \label{auxeq1}
	\left(1-\frac{1+2 A_fc_n^{\beta}+N_i^{-1}}{(N_i A_fc_n^{\beta})(1+ A_fc_n^{\beta})} \right)^{-1} \leq \left(1-\frac{2}{f_0A_f} \frac{1}{(\log n)^{1/3+\beta/3d} n^{2/3-\beta/3d}(1+o(1))} \right)^{-1} \quad \text{ for all }i,
	\end{equation}
	and, 
	\begin{equation} \label{auxeq2}
	\left(1-\frac{1+2 A_fc_n^{\beta}+N_i^{-1}}{(N_i A_fc_n^{\beta})(1+ A_fc_n^{\beta})} \right)^{-1} \geq \left(1-\frac{1}{2f_1A_f} \frac{1+o(1)}{(\log n)^{1/3+\beta/3d} n^{2/3-\beta/3d}} \right)^{-1} \quad\text{ for all }i.
	\end{equation}
	Let us prove that,
	\begin{equation} \label{auxeq3}
	\prod_{i=1}^{m_n}  \left(1-\frac{1+2 A_fc_n^{\beta}+N_i^{-1}}{(N_i A_fc_n^{\beta})(1+ A_fc_n^{\beta})} \right)^{-1} \stackrel{a.s.}\rightarrow 1.
	\end{equation}
	
	Since the right hand side of \eqref{auxeq1} can be express, for $n$ large enough, as
	$$\exp\Big(-C (\log n)^{1/3+\beta/3d} n^{2/3-\beta/3d}(1+o(1))\Big),$$
	being $C$ a positive constant, we get, for $n$ large enough,
	\begin{equation*}
	\prod_{i=1}^{m_n}  \left(1-\frac{1+2 A_fc_n^{\beta}+N_i^{-1}}{(N_i A_fc_n^{\beta})(1+ A_fc_n^{\beta})} \right)^{-1} \leq 
	\exp\Big(-Cm_n (\log n)^{1/3+\beta/3d} n^{2/3-\beta/3d}(1+o(1))\Big)\rightarrow 1,
	\end{equation*}
	where the limit follows from $\big| m_n (\log n)^{-1/3-\beta/3d} n^{-2/3+\beta/3d}\big|\leq (\log n)^{-\beta/3d} n^{-1/3+\beta/3d} \rightarrow 0$.
	\eqref{auxeq3} is obtained doing the same with \eqref{auxeq2}.
	
	Now let us study the asymptotic behaviour of $\prod_{i=1}^{m_n} \pr\left(N_i'\Delta^d(\mathcal{Y}_{N'_i})\leq \frac{w'_nN_i'}{a_in} \right)$. If we apply Lemma \ref{lem0}, (observe that the functions $a_-^S$ only depends on the shape of $S$) we get,
	$$
	\prod_{i=1}^{m_n} \pr\left(N_i'\Delta^d(\mathcal{Y}_{N'_i})\leq \frac{w_n'N_i'}{a_in} \right)\leq 
	\exp \left( -\sum_{i=1}^{m_n} N_i' \left( \frac{N_i'\omega'_n}{a_in}\right)^{d-1}\exp\left(-\frac{N_i'\omega'_n}{a_in}\right)a_-^{[0,1]^d}\left( \frac{N_i'\omega'_n}{a_in},N_i'\right) \right).
	$$
	
	Let, $\eps_{n,i}'=\frac{N_i}{a_in}\frac{w_n'}{w_n}-1$ for $i=1,\dots,m_n$. Observe that $\eps_{n,i}'=(1+\eps_{n,i})\frac{w_n'}{w_n}-1=(1+\eps_{n,i})\frac{1-A_fc_n^{\beta}}{1+2A_fc_n^{\beta}}-1$ and 
	$\eps_n'=\max_i|\eps_{n,i}'|$. Since $(\log(n))\eps_n \stackrel{a.s.}{\rightarrow} 0$, $\eps_n'$  fulfils
	$\log(n)\eps_n'\stackrel{a.s.}{\rightarrow}0$ and for all $i$, $N_i'\geq N_i$.
	The previous equation together with \eqref{t5i1} entails,
	
	\begin{equation}\label{t5i12}
	\pr(nV(\aleph_n)\leq w_n)\leq \exp\Big(-n\omega_n^{d-1}(1-\eps_n')^{d-1}\exp\big(-w_n(1+\eps_n')\big)a_-^{[0,1]^d}(\min\big(w_n(1-\eps_n'),\min_i N_i\big))\Big).
	\end{equation}
	
	If we choose $w_n=x+n\log(n)+(d-1)\log(\log(n))+\log(\alpha_A)$ in \eqref{t5i12} we get,
	\begin{equation}\label{t5i1fin}
	\pr(U(\aleph_n)\leq x)\leq \exp(-\exp(-x))+o(1).
	\end{equation}
	In order to conclude the proof of \eqref{conv0} we have to bound $\pr(U(\aleph_n)\leq x)$ from below. We provide just a sketch of the proof since the arguments are similar to those in Proposition \ref{theounif2}, using
	Lemma \ref{lem1} as in the proof of \eqref{t5i1fin}.
	Let us denote,
	
	\begin{itemize}
		\item $\rho_n=\frac{r_f \rho_A}{f_0^{1/d}}\big( \frac{\log(n)}{n}\big)^{1/d}$ with $\rho_A=\max_{x\in A}\|x\|$.
		\item $G_n=\cup_{i\neq j}^{m_n} (\overline{C_i}\cap \overline{C_j})$.
		
		\item $H_n=S\setminus (\cup_i^{m_n} C_i)$, notice that $H_n\subset \partial S^{c_n}$.
	\end{itemize}
	
	Proceeding as in Proposition \ref{theounif2} we have
	$$U(\aleph_n)\leq \max\Big\{\mathring{U}\big(\aleph_n\big), U(\aleph_n,G_n^{\rho_n}), U(\aleph_n,H_n)\Big\}\text{ eventually almost surely},$$
	and
	\begin{equation*}
	\pr(\mathring{U}(\aleph_n)\leq x)\geq \exp(-\exp(-x))+o(1).
	\end{equation*}
	
	Then, reasoning as in Proposition \ref{theounif2} we get
	\begin{equation*}
	\pr(U(\aleph_n)\leq x)\geq \exp(-\exp(-x))+o(1).
	\end{equation*}
	Finally, in order to conclude the proof of \eqref{conv0} it suffices 
	to prove that $G_n^{\rho_n}$ and $H_n$ satisfies the hypothesis of  Lemma \ref{lem02}.
	
	$G_n$ is the union of less than $m_n2^d$ $(d-1)-$dimensional cubes of size $c_n$. Each of them
	can be cover by less than $a_1 c_n^{d-1} \rho_n^{-d+1}$ balls of radius $\rho_n$ (with $a_1$ a positive constant), centered at some points $\{x^i_j\}_{i,j}$.
	Since  $G_n^{\rho_n}\subset
	\bigcup_{i,j} \mathcal{B}(x_j^i,2\rho_n)$, and every  $\mathcal{B}(x_j^i,2\rho_n)$ can be covered by 
	$a_2\rho_n^d n$ balls of radius $n^{-1/d}$,
	$G_n^{\rho_n}$ can be covered by less than $\nu_n=m_n a_1 c_n^{d-1}\rho_n^{-d+1}a_2\rho_n^d n=\mathcal{O}(n^{1-\frac{2}{3d}} (\log n)^{\frac{2}{3d}})$ 
	balls of radius $n^{-1/d}$.
	
	In the same way it can be proved that $H_n$ can be covered with $\mathcal{O}(n^{1-\frac{d+\kappa}{3d}}\log(n)^{\frac{d+\kappa}{3d}})$ 
	balls of radius $n^{-1/d}$. Indeed, cover $\partial S$ with $\mathcal{O}(c_n^{-\kappa})$ balls of radius $c_n$, apply triangular inequality
	to obtain that the union of the balls with the same centre but with a radius $c_n\sqrt{d}$ covers $\partial S^{\sqrt{d}c_n}$ and then covers $H_n$, finally cover
	every of these balls by $\mathcal{O}((c_nn^{1/d})^{-d})$ balls of radius $n^{-1/d}$.
	
	\subsubsection*{Proof of \eqref{conv2}}
	
	In Equation \eqref{t5i12}, let us choose $w_n=\log(n)+c\log(\log(n))$ with  $c<(d-1)$ and introduce $N_k=\lceil\exp(\sqrt{k})\rceil$, 
	like in equation (3.12) in  \cite{jan87}, we obtain for $k$ large enough, $\pr(N_kV(\aleph_{N_k})\leq w_{N_k})\leq \exp(-\alpha_A k^{\frac{d-1-c}{2}}/2)$ and, Borel-Cantelli Lemma 
	implies that, with probability one, for $k$ large enough $N_k V(\aleph_{N_k})\leq w_{N_k}$. The rest of the proof is exactly the same as in Lemma 5 in \cite{jan87}

	\subsubsection*{Proof of \eqref{conv3}:}
	For any $u>0$ let us introduce the sequence $r_n=\left(\frac{\log n +(d+1+u)\log(\log(n))}{n}\right)^{1/d}$ and $\eps_n=\frac{1}{\log(n)\log(\log(n))}$.
	Cover $S$ with $\nu_n\leq C_S \eps_n^{-d}r_n^{-d}$ balls of radius $\eps_n r_n$. Reasoning like in \cite{jan87} we get
	$$\pr\left( \frac{nV(\aleph_n)-\log(n)}{\log(\log (n))}\geq d+1+u\right)=\pr(\Delta_n\geq r_n)\leq \nu_n \left(1-r_n^d(1-\eps_n)^d(1-2K_fr_n\diam(A)) \right)^n,$$

	that implies,
	$$\pr\left( \frac{nV(\aleph_n)-\log(n)}{\log(\log (n))}\geq d+1+u\right)\leq C_S \frac{(\log(\log(n)))^d}{(\log(n))^{2+u+o(1)}}.$$
	
	From Borel-Cantelli Lemma, taking $N_k=\lceil\exp(\sqrt{k})\rceil$ it follows that, with probability one, for $k$ large enough
	$\frac{N_k V(\aleph_{N_k})-\log(N_k)}{\log(\log (N_k))}\leq d+1+u$. Now take $n\geq N_k$ and $n\leq N_{k+1}$ and suppose that
	$\frac{N_k V(\aleph_{N_k})-\log(N_k)}{\log(\log (N_k))}\leq d+1+u$. We have, 
	$$nV(\aleph_n)\leq \frac{N_{k+1}}{N_k}V(\aleph_{N_k})\leq \exp \left(\frac{1}{2\sqrt{k}}(1+o(1))\right)(\log(N_k)+(d+1+u)\log(\log(N_k))),$$ 
	that entails,
	$$nV(\aleph_n)\leq  \left(1+\frac{1}{2\log(n)}(1+o(1)))\right)(\log(n)+(d+1+u)\log(\log(n))).$$ 
	Finally for $n$ large enough,
	$$nV(\aleph_n)\leq  \log(n)+(d+1+u)\log(\log(n))+1,$$
	so that $$\frac{nV(\aleph_n)-\log(n)}{\log(\log(n))}\leq d+1+u +\frac{1}{\log(\log n)}, $$ 
	which concludes the proof of \eqref{conv3}.

\section{Appendix B}
\subsection{Proof of Theorem \ref{testconvteo}}

The proof make use of the following two propositions. The first one gives conditions under which the maximal-spacing of two compact sets are close. 
The second one shows that if the set  $S$ is not convex, then $\mathcal{R}(\mathcal{H}(S)\setminus S)>0$.

\begin{proposition} \label{prop1} Let $A$ and $B$ be bounded and non-empty subsets of $\mathbb{R}^d$. 
	If 
	$d_H(A,B)\leq \eps$ and $d_H(\partial A, \partial B)\leq \eps$. Then  
	$\big|\mathcal{R}(A)-\mathcal{R}(B)\big|\leq  2\eps.$

\begin{proof}
	First we introduce $A'=\big\{ x\in A: d(x,\partial A)>2\eps\big\}$ 
	and prove that $A'\subset B$ by contradiction.
	Suppose that there exists $x\in A$ such that $d(x,\partial A)>2\eps$ 
	and $x\notin B$. Since $d_H(A,B)\leq \eps$ we have $A\subset B^\varepsilon$, then $x\in B^\varepsilon \setminus B$, 
	so $d(x,\partial B)\leq \eps$. Now, as $d_H(\partial A, \partial B)\leq \eps$, by the triangular inequality,
	$d(x,\partial A)\leq 2\eps.$ which is a contradiction.
	From $A'\subset B$ it follows that 
	$\mathcal{R}(A')\leq \mathcal{R}(B)$.
	Now for all $r<\mathcal{R}(A)$ there exist $x\in A$ such that $\mathcal{B}(x,r)\subset A$ so that $\mathcal{B}(x,r-2\eps)\subset A'$
	which entails $\mathcal{R}(A')\geq \mathcal{R}(A)-2\eps$ and, finally, $\mathcal{R}(B)\geq \mathcal{R}(A)-2\eps$.
	Proceeding in the same way, we get  $\mathcal{R}(A)\geq \mathcal{R}(B)-2\eps$ that conclude the proof.
\end{proof}
\end{proposition}

\begin{proposition} \label{prop2} Let $S\subset \mathbb{R}^{d}$ be a non--convex, closed 
	set with non-empty interior. Then,  $\mathcal{R}\big(\mathcal{H}(S)\setminus S\big)>0.$

\begin{proof}
	Since $S$ is closed and non-convex there exists $x\in \mathcal{H}(S)\setminus S$ with $d(x,S)=r>0$.
	By Corollary 7.1 in \cite{gallier} we know 
	that $\mathcal{H}(S)=\overline{\mathring{\mathcal{H}(S)}}$ so that,
	for all $\eps>0$, there exists $x_{\eps}$ and $\nu_{\eps}>0$ such that $|x_{\eps}-x|\leq \eps$ and
	$\mathcal{B}(x_{\eps},\nu_{\eps})\subset \mathcal{H}(S)$. Taking $\eps=r/2$ and $\rho=\min(\nu_{r/2},r/2)>0$
	we conclude that $\mathcal{R}\big(\mathcal{H}(S)\setminus S\big) \geq \rho>0$.\\
\end{proof}
\end{proposition}

Now we can prove Theorem \ref{testconvteo}.\\

First observe that, if $S$ is convex $ \mathcal{R}\big(\mathcal{H}(\aleph_n)\setminus \aleph_n\big)\leq \mathcal{R}\big( S\setminus \aleph_n\big) $
and $|\mathcal{H}(\aleph_n)|\leq |S| $ so that  $\tilde{V_n}\leq V(\aleph_n)$. Then, from  Corollary \ref{jan2}
we obtain that $\mathbb{P}\big(\tilde{V_n}>c_{n,\gamma}\big)\leq \gamma +o(1)$, and  the 
test is asymptotically of level smaller than $\gamma$.\\

Now we prove that if $S\in \mathcal{C}_P$, then $\pr_{H_0}\big(\tilde{V}_n > c_{n,\gamma} \big)\rightarrow \gamma$.
Recall that,  if $S\subset \mathbb{R}^d$ is convex and $\aleph_n=\{X_1,\dots,X_n\}$ is an iid random sample,
uniformly drawn on $S \in \mathcal{C}_P$, in \cite{walther:97} it is proved that, almost surely: 

\begin{equation} \label{dhchull}
d_H\big(\mathcal{H}(\aleph_n), S\big)=\mathcal{O}\Big(\big(\log(n)/n\big)^{2/(d+1)}\Big) \text{ and } 
d_H\big(\partial \mathcal{H}(\aleph_n),\partial S\big)=\mathcal{O}\Big(\big(\log(n)/n\big)^{2/(d+1)}\Big).
\end{equation}

Thus, by Proposition \ref{prop1}, we have that
$\Big|\mathcal{R}\big(\mathcal{H}(\aleph_n)\setminus \aleph_n\big)-\mathcal{R}\big(S\setminus \aleph_n\big)\Big|=\mathcal{O}\Big(\big(\log(n)/n\big)^{2/(d+1)}\Big)$
almost surely. Therefore
$$\tilde{\Delta}_n(\aleph_n)=\frac{|\mathcal{H}(\aleph_n)|^{1/d}}{|S|^{1/d}}\left(\Delta(\aleph_n)+\mathcal{O}\Big(\big(\log(n)/n\big)^{2/(d+1)}\Big)\right)\text{ a.s.}$$

The second equation in \eqref{dhchull} also provides that $|\mathcal{H}(\aleph_n)|=|S|+\mathcal{O}\Big(\big(\log(n)/n\big)^{2/(d+1)}\Big)$ almost surely.
Finally we have,
\begin{equation*}
\tilde{V}_n=V(\aleph_n)\left(1+\mathcal{O} \Big(\big(\log(n)/n\big)^{2/(d+1)}\Big)\right)\left(1+\mathcal{O} \left( \frac{\big(\log(n)/n\big)^{2/(d+1)}}{\Delta(\aleph_n)}\right)\right)^d \text{ a.s.}
\end{equation*}

Observe that from Corollary \ref{jan2} $ii)$ and $iii)$ $\Delta(\aleph_n)=\left( \log(n)/n \right)^{1/d}(1+o(1))$ almost surely, then 
\begin{equation*}
\tilde{V}_n=V(\aleph_n)+\mathcal{O} \Big(\big(\log(n)/n\big)^{1+\frac{d-1}{d(d+1)}}\Big) \text{ a.s.}
\end{equation*}

This, together with $c_{n,\gamma}=\mathcal{O}(\log(n)/n)$ entails that $\pr(\tilde{V}_n\geq c_{n,\gamma})\rightarrow \gamma$, as desired.
\\

To conclude the proof of the theorem consider now that $S$ is not convex. First we prove that, if $\varepsilon_n=d_H(S,\aleph_n)$, then $d_H(\mathcal{H}(S),\mathcal{H}(\aleph_n))\leq 2\eps_n$. Indeed, for all $x\in \mathcal{H}(S)$ there exist $x_1 \in S$, $x_2\in S$
and $\lambda\in[0,1]$ such that $x=\lambda x_1+(1-\lambda)x_2$. Since  $\varepsilon_n=d_H(S,\aleph_n)$ there exist $X_i\in \aleph_n$ and $X_j\in \aleph_n$ such that
$\|x_1-X_i\|\leq \eps_n$ and $\|x_2-X_j\|\leq \eps_n$ so that $y=\lambda X_i+(1-\lambda)X_j$ belongs to 
$\mathcal{H}(\aleph_n)$. By the triangular inequality $\|x-y\|\leq 2\eps_n$. Since $\mathcal{H}(\aleph_n)\subset \mathcal{H}(S)$ 
we also have that  $d_H(\partial \mathcal{H}(S),\partial \mathcal{H}(\aleph_n))\leq 2\eps_n$,
which implies that  $\|\mathcal{H}(\aleph_n)|-|\mathcal{H}(S)\|\leq \mathcal{O}(\eps_n)$. On 
the other hand we have 
$d_H(\mathcal{H}(S)\setminus \aleph_n,\mathcal{H}(\aleph_n)\setminus \aleph_n)\leq 2\eps_n $
and 
$d_H(\partial (\mathcal{H}(S)\setminus \aleph_n),\partial (\mathcal{H}(\aleph_n)\setminus \aleph_n))\leq 2\eps_n$. By 
Proposition \ref{prop1} we get,
$\mathcal{R}(\mathcal{H}(\aleph_n)\setminus \aleph_n)\geq \mathcal{R} (\mathcal{H}(S)\setminus \aleph_n)-2\eps_n$. Since 
$\mathcal{H}(S)\setminus S\subset \mathcal{H}(S)\setminus \aleph_n$ it follows,
\begin{equation}\label{minonotconv}
\mathcal{R}(\mathcal{H}(\aleph_n)\setminus \aleph_n)\geq \mathcal{R} (\mathcal{H}(S)\setminus S)-2\eps_n.
\end{equation}

Since $\eps_n\rightarrow 0$ almost surely, 
$|\mathcal{H}(\aleph_n)|\stackrel{a.s.}{\rightarrow}|\mathcal{H}(S)|$ and
$\mathcal{R}(\mathcal{H}(\aleph_n)\setminus\aleph_n)\geq \mathcal{R}(\mathcal{H}(S)\setminus S)$ eventually almost surely.
Then, there exists $C_S$ a positive constant that depends on $S$ such that,
$\tilde{V}_n\geq C_S$ eventually almost surely.
Finally, since $c_{n,\gamma}\rightarrow 0$ we conclude the proof.

\subsection{Proof of Theorem \ref{mainnonconvtheo}}
The proof of Theorem \ref{mainnonconvtheo} is based on the following lemma

\begin{lemma}\label{newdens} Assume that the unknown density $f$ fulfils condition B and 
	$S\in\mathcal{A}$. Take $K\in\mathcal{K}$, and $h_n=\mathcal{O}(n^{-\beta})$ with $\beta\in(0,1/d)$.
	
	\noindent Let $\hat{f}_n(x)$ be the density estimator introduced in Definition \ref{defestimdens}. Then, 
	
	\begin{enumerate}
		\item[(i)] there exists a sequence $\eps^{+}_n$ such that $\log(n)\eps^{+}_n\rightarrow 0$ and for all $x\in S$, 
		$\left(\frac{f(x)}{\hat{f}_n(x)}\right)^{1/d}\geq 1-\eps^+_n$  e.a.s. 
		\item[(ii)] there exist a sequence $\eps^{-}_n\rightarrow 0$ and a constant $\lambda_0>0$ such that for all
		$x\in\mathcal{H}(\aleph_n)$, $(\hat{f}_n(x))^{1/d}\geq \lambda_0 -\eps^-_n.$  e.a.s. 
	\end{enumerate}

\begin{proof}

	We start the proof by establishing some useful preliminary results.
	First notice that, for $S\in\mathcal{A}$, with exactly the same kind of calculation we did to prove Lemma \ref{lem01}, choosing  $\rho_n=\left(\frac{4 \log n}{f_0c_S\omega_d n} \right)^{1/d}$
	we have,
	\begin{equation}\label{bounddh}
	\pr(d_H(\aleph_n,S)\geq \rho_n)\leq C_S n^{-2} \quad \text{for $n$ large enough}.
	\end{equation}
Notice that, since $K\in \mathcal{K}$, $S\in \mathcal{A}$, and $K$ is bounded from below on a neighbourhood of the origin, there exist $c_K''>0$
		and $r_K>0$ such that,
	\begin{equation}\label{standK}
	\int_{S} K((u-x)/r)du \geq  c''_K r^d  \quad  \text{for all } x\in S \text{ and } r\leq r_K'.
	\end{equation}
	
	We have, for all $x\in S$,
	$$\E f_n(x)=\int_{\{u:x+uh_n\in S\}}K(u)f(x+uh_n)du.$$
	Using that $f$ is Lipschitz and $\int_{\mathbb{R}^d} K(u)du=1$, we get, for all $x\in S$
	\begin{equation}\label{majdens}
	\E f_n(x)\leq  \int_{\{u:x+uh_n\in S\}}K(u)\big(f(x)+k_f\|u\|h_n\big)du\leq f(x)+k_fh_nc_K.
	\end{equation}
From \eqref{standK} and the condition $f(x)>f_0$ for all $x\in S$, it follows that,
	\begin{equation}\label{mindens}
	\E f_n(x)\geq f_0c'_K \quad \text{ for all } x\in S.
	\end{equation}
	
	We start by proving $(i)$. The triangular inequality entails that,
	\begin{equation} \label{ineqlem}
	\max_{x\in S}\big(\hat{f}_n(x)-f(x)\big)\leq \sup_{x\in S}\big|\hat{f}_n(x)-\E\hat{f}_n(x)\big|+\sup_{x\in S}\big(\E\hat{f}_n(x)-f(x)\big).
	\end{equation}
In order to deal with the first term on the right hand side of \eqref{ineqlem}, observe that, 
		since $K\in\mathcal{K}$ and $h_n=\mathcal{O}(n^{-\beta})$ with $\beta\in(0,1/d)$
		we can apply Theorem 2.3 in \cite{gg02}.
		Then, there exists a constant $C_1$ such that, with probability one, for $n$ large enough,
		$$\sqrt{\frac{nh_n^d}{-\log(h_n)}}\sup_{x\in \mathbb{R}^d}\big|f_n(x)-\E f_n(x)\big|\leq C_1.$$

		Thus
		$$\sqrt{\frac{nh_n^d}{-\log(h_n)}}\sup_{x\in \aleph_n} \big|f_n(x)-\E f_n(x)\big|\leq C_1,$$
		and therefore, 
		\begin{equation}\label{ineqnonunif1}
		\sqrt{\frac{nh_n^d}{-\log(h_n)}}\sup_{x\in S} \big|\hat{f}_n(x)-\E\hat{f}_n(x)\big|\leq C_1.
		\end{equation}

	Now let us bound $\sup_{x\in S} (\E \hat{f}_n(x)-f(x))$ from above. 
	For all $x\in S$ we have, 
	\begin{multline} \label{auxeqlem1}
	\E(\hat{f}_n(x))= \E(\hat{f}_n(x)|d_H(\aleph_n,S)\leq \rho_n)\pr(d_H(\aleph_n,S)\leq \rho_n)+\\
	\E(\hat{f}_n(x)|d_H(\aleph_n,S)> \rho_n)\pr(d_H(\aleph_n,S)> \rho_n).
	\end{multline}
	
	Since $\{(x,y)\in S^2,\|x-y\|\leq h_n\}$ is compact,
	the Lebesgue dominate convergence theorem entails that there exist $y_0\in S$ such that $\|x-y_0\|\leq \rho_n$, and a sequence $y_k$ with $y_k\rightarrow y_0$, $\|y_k-y_0\|\leq\rho_n$,  such that
	for $n$ large enough, with probability one
	
	\begin{multline*}
	\E(\hat{f}_n(x)| d_H(S,\aleph_n)\leq \rho_n)\leq  \sup_{x\in S}\E\Big(\limsup_{y\in S:\|x-y\|\leq \rho_n}f_n(y)\Big)=\\
	\sup_{x\in S}\E\Big(\lim_{y_k\rightarrow y_0} f_n(y_k)\Big)=\sup_{x\in S}\lim_{y_k\rightarrow y_0} (\E(f_n(y_k)))\leq\sup_{x\in S} \sup_{y\in S:\|x-y\|\leq\rho_n} \E(f_n(y)).
	\end{multline*}
	Now applying \eqref{majdens} and the Lipschitz continuity of $f$ we obtain,
	\begin{multline}\label{majo51}
	\E(\hat{f}_n(x)| d_H(S,\aleph_n)\leq \rho_n)\leq \max_{y\in S,\|x-y\|\leq \rho_n} \{f(y)+k_f h_nc_K \}
	\leq f(x)+k_f \rho_n +k_f h_nc_K.
	\end{multline}
	
	With the same kind of argument it can be proved that,
	
	\begin{equation}\label{majo52}
	\E(\hat{f}_n(x)|d_H(\aleph_n,S)\geq \rho_n)\leq \sup_{y\in S} \E f_n(y) \leq f_1+k_f h_nc_K.
	\end{equation}
	
	From equations \eqref{auxeqlem1},\eqref{majo51},\eqref{majo52} and \eqref{bounddh} we get,
	\begin{equation}\label{finish}
	\sup_{x\in S}(\E(\hat{f}_n(x))-f(x))\leq k_f \rho_n +k_f h_nc_K +(f_1+k_f h_nc_K)C_Sn^{-2}.
	\end{equation}

	Take now $\eps_n=k_f \rho_n +k_f h_nc_K + (f_1+k_f h_nc_K)C_Sn^{-2}+C_1\left(\frac{nh_n^d}{-\log(h_n)}\right)^{1/2}$. Which fulfils  $\log(n)\eps_n\rightarrow 0$.
	From equations \eqref{ineqlem}, (\ref{ineqnonunif1}), (\ref{finish}), we obtain that, with probability one, for $n$ large enough, 
	$$\max_{x\in S}\big(\hat{f}_n(x)-f(x)\big)\leq \eps_n.$$

	Then, for all $x\in S$, $\hat{f}_n(x)-f(x)\leq f(x)\eps_n/f_0$, 
	and thus, $\frac{\hat{f}_n(x)}{f(x)}\leq 1+\frac{\eps_n}{f_0}$,
	or equivalently, $$\left(\frac{f(x)}{\hat{f}_n(x)}\right)^{1/d}\geq \left(1+\frac{\eps_n}{f_0}\right)^{-1/d} .$$
	Finally,  taking  $\eps_n^+=(1-(1+\eps_n/f_0)^{-1/d})\sim \eps_n/(df_0)$ (observe that we have $\eps_n^+\log(n)\rightarrow 0$)
	then $\max_{x\in S}\left(\frac{f(x)}{\hat{f}_n(x)}\right)^{1/d}\geq 1-\eps_n^+$ eventually almost surely, which concludes the proof of $(i)$.\\

	In order to prove $(ii)$, observe that
	
	$$\min_{x\in \mathbb{R}^d}\hat{f}_n(x)\geq \min_{x\in \mathbb{R}^d}\E \hat{f}_n(x) -\max_{x\in \mathbb{R}^d} \big|\E \hat{f}_n(x)-\hat{f}_n(x)\big|.$$
	Since we have already proved that $\max_{x\in \mathbb{R}^d} \big|\E \hat{f}_n(x)-\hat{f}_n(x)\big|\rightarrow 0 \text{ a.s.}$, it remains to prove that $\min_{x\in \mathbb{R}^d}\E \hat{f}_n(x)$ is bounded from below by a positive constant.
	From $\min_{x\in \mathbb{R}^d}\E \hat{f}_n(x)=\min_{x\in\aleph_n}\E f_n(x)$ and \eqref{mindens}, we get
	$$\min_{x\in \mathbb{R}^d}\E \hat{f}_n(x)\geq \min_{x\in S}\E f_n(x)\geq f_0c'_K.$$
	\end{proof}
	
\end{lemma}
Now we are ready to prove Theorem \ref{mainnonconvtheo} .
\subsection*{Proof of Theorem \ref{mainnonconvtheo} a)}
Recall that 
$$\hat{\delta}\big(\mathcal{H}(\aleph_n)\setminus \aleph_n\big)=\sup \bigg\{r: \ \exists x \text{ such that } x+\frac{r}{\hat{f}_n(x)^{1/d}}A\subset \mathcal{H}(\aleph_n)\setminus \aleph_n \bigg\},$$ with $A$ the ball $\mathcal{B}(O,\omega_d^{-1/d})$.

Under $H_0$ ($S$ is convex),

$$\hat{\delta}\big(\mathcal{H}(\aleph_n)\setminus \aleph_n\big)\leq \sup \bigg\{r: \ \exists x \text{ such that } x+\frac{r}{\hat{f}_n(x)^{1/d}}A\subset S\setminus \aleph_n \bigg\}.$$

If we apply Lemma \ref{newdens} (i), (notice that all convex sets are in $\mathcal{A}$)  we get,

$$\hat{\delta}\big(\mathcal{H}(\aleph_n)\setminus \aleph_n\big)\leq \sup \bigg\{r: \ \exists x \text{ such that } x+\frac{r}{f(x)^{1/d}}(1-\eps_n^+)A\subset S\setminus \aleph_n \bigg\}.$$

Equivalently, $\Delta(\aleph_n)\geq (1-\eps_n^+)\hat{\delta}\big(\mathcal{H}(\aleph_n)\setminus \aleph_n\big),$ and therefore
$\pr(\hat{V}_n\geq c_{n,\gamma})\leq \pr\big(V(\aleph_n) \geq (1-\eps_n^+)^dc_{n,\gamma}\big),$
from where it follows that $\pr(\hat{V}_n\geq c_{n,\gamma})$ can be majorized by,
$$ \pr\Big(U(\aleph_n) \geq -(1-\eps_n^+)^d\log\big(-\log(1-\gamma)\big)+\big((1-\eps_n^+)^d-1\big)\big(\log(n)+(d-1)\log(\log(n))+
\log(\alpha_{\mathcal{B}})\big)\Big).$$

Therefore,  by Theorem \ref{theocont}, (using that $\log(n)\eps_n^+\rightarrow 0$)  we get that,

$$\pr(\hat{V}_n\geq c_{n,\gamma})\leq \pr\big(U(\aleph_n) \geq -\log(-\log(1-\gamma))+o(1)\big)\rightarrow \gamma . $$

\subsection*{Proof of Theorem \ref{mainnonconvtheo} b)}

From Lemma \ref{newdens} (ii), we have that
$$\hat{\delta}\big(\mathcal{H}(\aleph_n)\setminus \aleph_n\big)\geq (\lambda_0 -\eps^-_n)
\mathcal{R}\big(\mathcal{H}(\aleph_n)\setminus \aleph_n\big),$$
where $\eps_n^-\rightarrow 0$ a.s.
Then, under $H_1$ ($S$ is not convex), from \eqref{minonotconv}, we obtain
$$\hat{\delta}\big(\mathcal{H}(\aleph_n)\setminus \aleph_n\big)\geq (\lambda_0 -\eps^-_n)\big(\mathcal{R}\big(\mathcal{H}(S)\setminus S\big)-2d_H(S,\aleph_n)\big).$$
Since $S\in \mathcal{A}$, $d_H(S,\aleph_n)\rightarrow 0$ a.s. (see \cite{crc:04}), and $\mathcal{R}\big(\mathcal{H}(S)\setminus S\big)>0$ (see Proposition \ref{prop2}) then
		with probability one, for $n$ large enough
		$$\hat{\delta}\big(\mathcal{H}(\aleph_n)\setminus \aleph_n\big)\geq \frac{1}{2}\lambda_0\mathcal{R}\big(\mathcal{H}(S)\setminus S\big),$$
		and Theorem \ref{mainnonconvtheo} b) follows from the fact that $c_{n,\gamma}\rightarrow 0$.

\end{document}